\documentclass[aap,preprint]{imsart}
\usepackage[utf8]{inputenc}
\usepackage{physics}
\usepackage{xcolor}
\usepackage{tcolorbox}
\usepackage{bbm}
\usepackage{cancel}
\usepackage{mathrsfs}
\usepackage{amsthm}
\usepackage{algorithmic}
\usepackage{mathtools}
\usepackage{hhline}
\usepackage{tabularx}
\usepackage[ruled,vlined]{algorithm2e}
\usepackage{hyperref}

\RequirePackage{amsthm,amsmath,amsfonts,amssymb}
\RequirePackage[numbers]{natbib}

\startlocaldefs

\theoremstyle{plain}
\newtheorem{thm}{Theorem}
\newtheorem*{thm*}{Theorem}
\newtheorem{lem}[thm]{Lemma}
\newtheorem{corollary}[thm]{Corollary}
\newtheorem{proposition}[thm]{Proposition}
\theoremstyle{remark}
\newtheorem{assm}[thm]{Assumption}
\newtheorem{rem}[thm]{Remark}
\renewcommand{\norm}[1]{\left\Vert{#1}\right\Vert}
\newcommand{\Om}[0]{\Omega}
\newcommand{\sL}[0]{\mathscr{L}}

\newcommand{\pa}[1]{\left( {#1} \right)}
\newcommand{\an}[1]{\left\langle {#1}\right\rangle}
\newcommand{\ba}[1]{\left[ {#1} \right]}
\newcommand{\ab}[1]{\left| {#1} \right|}
\newcommand{\bc}[1]{\left\{ {#1} \right\}}
\newcommand{\ve}[1]{\left\| {#1} \right\|}
\newcommand{\fc}[2]{\frac{#1}{#2}}
\newcommand{\pf}[2]{\pa{\frac{#1}{#2}}}
\newcommand{\rc}[1]{\frac{1}{#1}}
\newcommand{\Var}[0]{\operatorname{Var}}
\newcommand{\ddd}[1]{\frac{d}{d #1}}

\newcommand{\TV}[0]{\operatorname{TV}}
\newcommand{\KL}[0]{\operatorname{KL}}
\newcommand{\be}[0]{\beta}
\newcommand{\flipi}[2]{{#1}^{\oplus #2}}
\newcommand{\Ent}[2]{\operatorname{Ent}_{#2}[{#1}]}

\newcommand{\blu}[1]{\textcolor{black}{#1}}

\begin{document}

\begin{frontmatter}
\title{Convergence Bounds for Sequential Monte Carlo\\on Multimodal 
Distributions using Soft Decomposition}

\begin{aug}
\author[A]{\fnms{Holden}~\snm{Lee}\ead[label=e1] {hlee283@jhu.edu}}
\and
\author[A]{\fnms{Matheau}~\snm{Santana-Gijzen}\ead[label=e2]{msanta11@jhu.edu}}
\address[A]{Department of Applied Mathematics \& Statistics, Johns Hopkins University\printead[presep={,\ }]{e1,e2}}

\end{aug}

\begin{abstract}
  We prove bounds on the variance of a function $f$ under the empirical measure of the samples obtained by the Sequential Monte Carlo (SMC) algorithm, with time complexity depending on local rather than global Markov chain mixing dynamics. SMC is a Markov Chain Monte Carlo (MCMC) method, which starts by drawing $N$ particles from a known distribution, and then, through a sequence of distributions, re-weights and re-samples the particles, at each instance applying a Markov chain for smoothing. 
    In principle, SMC tries to alleviate problems from multi-modality. However, 
    most theoretical guarantees for SMC are obtained by assuming global mixing time bounds, which are only efficient in the uni-modal setting. We show that bounds can be obtained in the truly multi-modal setting, with mixing times that depend only on local MCMC dynamics. Prior works in the multi-modal setting  \cite{schweizer2012nonasymptotic,paulin2018error,mathews2022finite}
    make the restrictive assumption that the Markov kernel in the SMC algorithm has no or limited movement 
    between disjoint partitions of the state space. Instead, we allow for mixing between modes by only considering SMC with Markov processes whose generator decomposes as $\langle f, \mathscr{L}_kf \rangle_{\mu_k} \leq \sum_{i=1}^m w_i\langle f, \mathscr{L}_{ki}f \rangle_{\mu_k^{(i)}}$. For mixture distributions, this covers common Markov chain approaches such as Langevin dynamics, Glauber dynamics and the Metropolis random walk. \blu{The main technical innovations in our proof are showing variance decay and hypercontractivity for mixture distributions. We apply our results to show efficient sampling from Gaussian mixtures.}
\end{abstract}



\end{frontmatter}


\section{Introduction}
Obtaining samples from a given target distribution of the form $p(x) \propto e^{-V(x)}$ is an important task across multiple disciplines. 
In higher dimensions, Markov Chain Monte Carlo (MCMC) methods are a useful tool in obtaining samples from a specified target. The idea is that given a number of particles, running a Markov chain to equilibrium will yield samples from the chosen stationary distribution. This generally works well for uni-modal distributions. However, multi-modal distributions pose serious issues for mixing times of MCMC methods. Specifically, the long time needed to move between modes can yield very poor mixing times. 

Sequential Monte Carlo (SMC) methods attempt to circumvent this issue by preserving the appropriate amount of samples in each mode. The general idea is that SMC obtains samples from a target distribution $\mu = \mu_n$ by moving particles through a sequence of distributions $\mu_0,\dots,\mu_n$. Particles are moved from $\mu_{k-1}$ to $\mu_k$ and then re-sampled according to their respective weights. After re-sampling at each level, a Markov chain is run 
so that the particles are more representative of the stationary distribution $\mu_k$.

SMC methods are well studied and there exists a body of work providing asymptotic and non-asymptotic guarantees \cite{schweizer,Chopin_2004}. However, much of the existing work only establishes bounds 
given global mixing times. In the multi-modal setting, these bounds are only relevant if the Markov chain is run for long times. 
Moreover, any bound depending on global mixing times does not improve upon bounds for vanilla MCMC methods targeting a single distribution, and hence does not explain the observed advantage of \emph{sequential} Monte Carlo for multi-modal distributions.
This issue is addressed by \cite{schweizer2012nonasymptotic,mathews2022finite}; however, both papers require a restrictive assumption that MCMC kernels cannot move samples between disjoint sets of the state space. 
By leveraging meta-stable approximations, \cite{paulin2018error} obtain a bound using local MCMC dynamics that allows movement between components. However, this approach requires very specific run times to ensure that particles only mix to the interior of a mode, thus limiting its application.

Our work builds upon existing literature by providing non-asymptotic variance for a non-negative function $f$ under the empirical measure $\mu_n^N$ which only depend on local mixing, and allow for arbitrarily overlapping support of mixture components.
Our results apply to SMC where the chosen kernel has the property that 
its generator satisfies the decomposition $\langle f, \mathscr{L}_kf \rangle_{\mu_k} \le \sum_{i=1}^m w_i\langle f, \mathscr{L}_{ki}f \rangle_{\mu_k^{(i)}}$, where $\mathscr{L}_i$ is the generator over the ith mixture component. It is easy to find popular examples of Markov chains which exhibit this property, including Langevin Dynamics and Metropolis Random Walk (MRW).

There are various other methods to approach sampling from multi-modal distributions. Closely related to SMC, simulated tempering \cite{Marinari_1992} moves particles between a sequence of distributions, allowing particles to move both directions through the sequence, and at each step applying a Markov kernel at each level. Guarantees  for sampling with simulated tempering using local mixing of Langevin Dynamics are provided in \cite{ge2020simulated}. 
Other approaches include parallel tempering \cite{woodard2009conditions} and annealed importance sampling \cite{neal2001annealed}. 

\blu{
A primary advantage of SMC compared to tempering methods is that we only have to move samples through distributions in one direction. However, this makes SMC more challenging to analyze, as the algorithm can no longer be understood as a single Markov chain, and so Markov chain decomposition theorems \cite{madras2002markov} do not apply. We will prove results for SMC under similar, albeit slightly stronger, assumptions on multimodal distributions that have been considered for simulated tempering \cite{ge2020simulated}.
}

\subsection{Overview of Sequential Monte Carlo}
Given a measure $\mu_n$ with density $p_n(x) \propto e^{-V(x)}$, Monte Carlo methods seeks to estimate $\mu_n(f)$ by $\frac{1}{N}\sum_{i=1}^N f(\xi_n^{(i)})$, where $\xi_n^{(i)} \sim \mu_n$ for $1\le i\le N$. 
SMC achieves this by first sampling  $\{\xi_1^{(i)}\}_{i=1}^N \sim \mu_1$, where $\mu_1$ in practice is chosen to be a measure for which samples are easily obtainable. The $N$ particles are then transformed through a sequence of measures $\mu_1, \dots, \mu_n$, resulting in (approximate) samples from $\mu_n$. In the transformation process, each particle is re-weighted according to the ratio $\frac{d\mu_{k}}{d\mu_{k-1}}(\xi_{k-1}^{(i)})$, which will be denoted as $\bar{g}_{k-1,k}(\xi_{k-1}^{(i)})$. The particles are then re-sampled with their respective weights and a Markov kernel is applied for smoothing. 

\begin{algorithm}
    \caption{Sequential Monte Carlo}
    \label{alg}
\begin{algorithmic}[1]
    \STATE Sample $\{\xi_1^{(i)}\}_{i=1}^N \sim \mu_1$. 
    \FOR{$k=1, \dots, n-1$}
    \STATE Re-sample $\{\hat{\xi}_{k+1}^{(i)}\}_{i=1}^N$ from a multinomial distribution with probabilities $\{w_{k+1}^{(i)}\}_{i=1}^N$ where
    $$w_{k+1}^{(i)} = \frac{\bar{g}_{k,k+1}(\xi_k^{(i)})}{\sum_{j=1}^N \bar{g}_{k,k+1}(\xi_k^{(j)}) }.$$
    Let $\hat{\xi}_{k+1}^{(i)} = \xi_{k}^{(i')}$ if index $i'$ is chosen.
    \STATE Obtain $\{{\xi}_{k+1}^{(i)}\}_{i=1}^N$ by applying the transition kernel $P_{k+1}$ to $\{\hat{\xi}_{k+1}^{(i)}\}_{i=1}^N$.
    \ENDFOR
    \STATE 
    Output $\{\xi_n^{(i)}\}_{i=1}^N$ as approximate samples from $\mu_n$. 
\end{algorithmic}
\end{algorithm}
\subsection{Notation}
Our work builds upon Schweizer's non-asymptotic bounds for Sequential MCMC Methods \cite{schweizer2012nonasymptotic} and by nature we adopt their notation. Suppose that $\mu_k$ has density $p_k\propto \tilde p_k$ and we have access to $\tilde p_k$. We denote the non-normalized ratio between two distributions by $g_{k-1,k}(x) = \fc{\tilde p_k(x)}{\tilde p_{k-1}(x)}$  and the normalized ratio by $\bar{g}_{k-1,k}(x) = \fc{p_k(x)}{p_{k-1}(x)}$. For a measure space $(\Om, \mu)$ we denote the integral by
\begin{equation*}
    \mu(f) \vcentcolon = \int_\Omega f(x) \mu(dx).
\end{equation*}
Given samples $\xi_1, \dots, \xi_n$ we denote the Monte Carlo estimate of $\mu(f)$ by
\begin{equation*}
    \eta^N(f) \vcentcolon = \frac{1}{N}\sum_{i=1}^N f(\xi_i).
\end{equation*}
To avoid the dependence in the sampling step, the analysis in Schweizer is defined through a filtration. 
The filtration is defined as $
\mathcal F_1\subseteq\cdots \subseteq \mathcal F_n$ where $\mathcal{F}_k$ is the $\sigma$-algebra generated by the samples $\{\xi_k^{(i)}\}_{i=1}^N$. Thereby, they 
define the following weighted empirical measure,
\begin{equation*}
    \nu(f) \vcentcolon = \varphi_k\nu_k^N(f), \text{ where } \varphi_k = \prod_{j=1}^{n-1}\eta_j^N(g_{j,j+1}).
\end{equation*}
For a measurable space $E$, let $B(E)$ be the set of bounded, measurable functions on $E$. 
For $f\in B(E)$, the action of a Markov kernel $K$ is defined as
\begin{equation*}
    K(f)(x) \vcentcolon= \int_Ef(z)K(x,dz)
\end{equation*}
and the measure $\mu K$ is defined as
\begin{equation*}
    \mu K(A) \vcentcolon = \int_E K(x,A)\mu(dx).
\end{equation*}
Lastly, to work with the filtration, we define the following:
\begin{equation*}
q_{k-1,k}(f) \vcentcolon= \bar{g}_{k-1,k}P_k(f).
\end{equation*}
The above can be interpreted as a single step of scaling the weights from $\mu_{k-1}$ to $\mu_k$ and then mixing at the $k$-th level. Similarly, the expression below can be interpreted as mixing at the $k-1$ level and then scaling to the $k$-th level: 
\begin{equation*}
    \hat{q}_{k-1,k}(f) \vcentcolon= P_{k-1}( \bar{g}_{k-1,k}f) .
\end{equation*}
Define $q_{j,k} = q_{j,j+1}\circ \cdots \circ q_{k-1,k}$ and similarly define $\hat q_{j,k}$.
The following properties of the above functions hold: 
\begin{align*}
    q_{j,k}(q_{k,n}(f)) &= q_{j,n}(f)&
    \hat{q}_{j,k}(\hat{q}_{k,n}(f)) &= \hat{q}_{j,n}(f)\\
\mu_j(q_{j,n}(f)) &= \mu_n(f) &
\mu_j(\hat{q}_{j,n}(f)) &= \mu_n(f).
\end{align*}
For a discrete time Markov chain $P$ we note that the Dirichlet form is defined as
\begin{equation}\label{dirichlet:discrete}
\mathscr{E}_\pi(f,f) \vcentcolon= \langle f , (I-P)f \rangle_\pi.
\end{equation}

\subsection{Discrete and Continuous-Time Generators}

\blu{A central assumption of our paper (Assumption \ref{a}) relies on properties of a Markov generator $\mathscr{L}$ and its associated Dirichlet form $\mathcal{E}(f,f) = -\langle f, \mathscr{L} f \rangle_{\mu}$. The building blocks for our main result, Lemma \ref{intra}
(bounding intra-mode variance) and Lemma \ref{hyper}
(hypercontractivity), are presented in a continuous-time framework.}

To show that our main results can be applied to a variant of common discrete-time Markov chains, namely Glauber dynamics (Section \ref{s:Glauber_dynamics}) and the Metropolis-Hastings chain (Section \ref{s:Metropolis_chain}), we formally connect the Markov generator of a discrete chain to its continuous time equivalent.

For any discrete-time, reversible Markov chain $P$ with stationary distribution $\pi$, we can define its associated \textbf{continuous-time Markov semigroup} $P_t$ by:
\begin{equation} \label{eq:ct_markov_semigroup}
P_t = e^{t(P-I)}
\end{equation}
The generator $\mathscr{L}$ of the continuous time equivalent is defined as: 
$$\mathscr{L}f = \frac{\partial}{\partial t}P_t f|_{t=0} = (P-I)f.$$

\begin{itemize}
    \item \textbf{Discrete-Time:} The Dirichlet form for the chain $P$ is what we have already defined in Equation (\ref{dirichlet:discrete}): 
    \begin{equation*}
    \mathcal{E}_{P}(f,f) := \langle f,(I-P)f\rangle_{\pi}.
    \end{equation*}

    \item \textbf{Continuous-Time:} The Dirichlet form for the generator $\mathscr{L}$ is:
    \begin{equation*}
    \mathcal{E}_{\mathscr{L}}(f,f) := -\langle f,\mathscr{L}f\rangle_{\pi} = -\langle f,(P-I)f\rangle_{\pi} = \langle f,(I-P)f\rangle_{\pi}.
    \end{equation*}
\end{itemize}

As shown, the two Dirichlet forms are identical:
\begin{equation} \label{eq:df_equiv}
\mathcal{E}_{P}(f,f) = \mathcal{E}_{\mathscr{L}}(f,f).
\end{equation}
This means that to show Assumptions \ref{a} hold, it is equivalent to work with the generator $\mathscr{L} = P-I$. 
Moreover, it should be noted that the continuous-time process $P_t$ can be simulated algorithmically by running the discrete-time process ``jump chain'' and drawing the holding time between each transition according to a Poisson point process \cite{Levin_Peres_Wilmer_Propp_Wilson_2017}. 
This makes our analysis not only theoretically sound but algorithmically relevant as well.
\color{black}
\subsection{Organization of paper}
In Section \ref{s:main}, we present the main results of our paper. Here we state our main assumptions and provide non-asymptotic bounds on the empirical variance $\eta_n^N(f)$. As corollaries to this result we provide both total variance and high probability bounds on the empirical estimate. Following this section, Section \ref{s:applications} provides simple applications of our main theorem. We show that our assumptions are satisfied by Langevin diffusion \blu{for mixtures of Gaussians of equal covariance, as well as sequences of Gaussian-noised distributions which can be learned with Noise Contrastive Estimation (NCE).}

In Section \ref{s:Overview}, we provide an overview of the proof of our main results. This section is designed to give readers an understanding of how the proceeding sections, Sections 5--8, are used in the proof of the main result. Here we explain how we connect the existing work in \cite{schweizer} to our results. In the following section, Section \ref{Var bounds}, we include the main results from \cite{schweizer}, which we adapt slightly to better fit our problem. 

Sections 6--8 contain our contributions to this problem. In Section \ref{local mixing}, we decompose the variance into intra- and inter-mode variances. This allows for us to use local Poincar\'e inequalities to bound the intra-mode variance. Furthermore, we are able to bound the inter-mode variance by assuming that our target measure is a mixture and inheriting a minimum mode weight term. Section \ref{mixing times} shows how the results in the previous section can be used to control the mixing time. Lastly, in Section \ref{s:hypercontractivity}, we show that the hypercontractive inequality over mixture distributions can be reduced to local mixing times by inheriting a minimum mode weight constant.

\section{Main Result} \label{s:main}

Before giving our main theorem, we first state our assumptions which capture how the distributions in our sequence decompose into multiple modes. These assumptions allow the state space to be either continuous or discrete, as they work for any Markov process whose generator decomposes over the components, with each satisfying a log-Sobolev inequality. We further need to assume that the ratio between subsequent measures is bounded.


\begin{assm}\label{a}
    Suppose SMC (Algorithm~\ref{alg}) is run with the sequence of distributions $\mu_1,\ldots, \mu_n$ with kernels $P_k = P_k^{t_k}$, where $P_k^t$ is the Markov semigroup generated by $\mathscr{L}_k$
    with stationary distribution $\mu_k$ and Dirchlet form $\mathscr{E}_k(f,f) = -\langle f, \mathscr{L}_kf\rangle_{\mu_k}$. In addition, assume the following hold: 
\begin{enumerate}
    \item \blu{At each level $k$, assume $\mu_k$ decomposes as the mixture $\mu_k = \sum_{i=1}^{\blu{M_k}}w_k^{(i)}\mu_k^{(i)}$, where $M_k$ is the number of components and $w_k^{(i)} > 0$ are the weights satisfying $\sum_{i=1}^{M_k}w_k^{(i)}=1$.}
    \item There exists $\gamma$ such that $\frac{d\mu_k}{d\mu_{k-1}} \leq \gamma$ a.e. for all $k$ (uniform upper bound on density ratios).
    \item There exists a decomposition 
$$\langle f, \mathscr{L}_kf \rangle_{\mu_k} \leq \sum_{i=1}^{\blu{M_k}} \blu{w_k^{(i)}}\langle f, \mathscr{L}_{ki}f \rangle_{\mu_k^{(i)}}$$
where $\mathscr{L}_{ki}$ is the generator of some Markov process $P_{ki}$ with stationary distribution $\mu_k^{(i)}$. 
\item Each distribution $\mu_k^{(i)}$ satisfies a log-Sobolev inequality of the form
$$\Ent{f^2}{\mu_k^{(i)}}\leq c_{ki} \cdot\mathscr{E}_{ki}(f,f),$$
and $C^*_k = \max_i c_{ki}$.
\end{enumerate}
\end{assm}

\blu{
\begin{rem}
    We note that our main theorem requires a bound on the minimum weight $w_k^{(i)}$, but does not assume any structure between the components; in particular, there is no requirement of closeness between components at different levels as we do not need any correspondence between them, and even the number of components can change. We leave as an open question how to derive finer bound when we take into account structure between the components. Intuitively, if a mode has small weight, this should only increase the complexity of SMC if this weight increases between levels. When the modes correspond across levels, then a natural conjecture is that the true complexity should depend not on the minimum weight but either a ``bottleneck coefficient'' or ``total variation'' in the weights, similarly to analyses of simulated tempering (see \cite[(7)]{woodard2009conditions}).
\end{rem}
We follow a common way of formalizing multimodality in the literature \cite{ge2018beyond,pmlr-v291-koehler25a} through the mixture assumption (part 1 above) where each component $\mu_k^{(i)}$ satisfies a functional inequality (as in part 4); that is, the natural Markov chain for the components mixes rapidly. We note that crucially, we only assume the existence of this decomposition, and its form is not known to the algorithm. As heuristic support for the decomposition hypothesis, we note it is polynomially equivalent to having at most $k$ low-lying eigenvalues for discrete state space \cite{lee2014multiway} or for Langevin \cite{miclo2015hyperboundedness}, and the latter holds for a distribution with $k$ modes in the limit of low temperature under mild regularity assumptions \cite[Chapter 8, Propositions 2.1--2.2]{kolokoltsov2007semiclassical}. We note that we do require the decomposition at each level $k$; this avoids cases where the target distribution may decompose well but bottlenecks are present at intermediate levels. An example where these bottlenecks arise is a mixture of Gaussians with different covariances (and perturbations thereof), for which the mixing time of parallel or simulated tempering can be exponentially large \cite{woodard2009sufficient}, and in fact, any algorithm that samples from this family in a black-box setting requires exponentially many queries in general
\cite{ge2018beyond}. We will give concrete examples where the assumptions are satisfied in Section \ref{s:applications}.
}

Our main result states that under these assumptions, for any $\epsilon>0$, there exists a number of samples $N$ and a time $t$ polynomial in all parameters, such that the SMC algorithm returns $N$ samples for which the empirical variance for a bounded function $f$ is within $\epsilon$ of the true variance. As corollaries to our main result, we will present both a high probability bound on $|\frac{1}{N}\sum_if(\xi_n^{(i)}) - \mu_n(f)|$ and a total variation bound on the empirical measure.

\begin{thm}\label{t:mainVar}
Suppose Assumption~\ref{a} holds.
Then for all $\epsilon > 0$, \blu{$M = \max_k M_k$}, \blu{$w^* = \min_{k,i} w_{k}^{(i)}$}, choosing
\blu{\begin{enumerate}
        \item $ N \geq n\cdot\max\bigg\{\frac{4\gamma M}{ w^*\epsilon}\cdot(1 
        + 3\Vert f - \mu_n(f)\Vert ^2_{\sup}),  \; \frac{128\gamma^{\frac{35}{8}}M^\frac{7}{4}}{(w^*)^\frac{15}{8}}\bigg\}$
        \item $t_k \geq 2C_k^*\gamma^7$
        \end{enumerate}}
yields 
    \begin{equation*}
        \blu{\mathbb{E}[(\eta_n^N(f) - \mu_n(f))^2]}\leq \epsilon.
    \end{equation*}
\end{thm}

\begin{proof} Proof provided in Section \ref{s:proofs}.
\end{proof}
This result is analogous to the main result in \cite{schweizer}, except that by specifying our target measure as a mixture we significantly reduce the overall mixing time. We will now show that as a corollary to our main result we can obtain high probability bounds similar to \cite{mathews2022finite}.

\begin{corollary}
[\bf High Probability Bound] Suppose Assumption~\ref{a} holds. Let
$\{\xi_n^{(1)},\dots, \eta_1^{(N)}\} \sim \hat{\mu}_n^N$ be the particles produced by the SMC algorithm. 
\blu{Then for all $\epsilon > 0$, \blu{$M = \max_k M_k$}, \blu{$w^* = \min_{k,i} w_{k}^{(i)}$}, choosing}
\blu{\begin{enumerate}
        \item $ N \geq n\cdot\max\bigg\{\frac{4\gamma M}{ w^*\epsilon^2\delta}\cdot(1 
        + 3\Vert f - \mu_n(f)\Vert ^2_{\sup}),  \; \frac{128\gamma^{\frac{35}{8}}M^\frac{7}{4}}{(w^*)^\frac{15}{8}}\bigg\}$
        \item $t_k \geq 2C_k^*\gamma^7$
        \end{enumerate}}
    yields 
    \begin{equation*}
        \ab{\frac{1}{N}\sum_{i=1}^N f(\xi_i^{(n)}) - \mu_n(f)} \leq \epsilon
    \end{equation*}
    with probability as least $1 - \delta$.
\end{corollary}
\begin{proof}
    The proof follows from Chebyshev's inequality, 
    \begin{align*}
        1 - \frac{\Var(\eta_n^N(f))}{\epsilon^2} &\leq P\pa{\ab{\frac{1}{N}\sum_{i=1}^N f(\xi_i^{(n)}) - \mu_n(f)} \leq \epsilon}.\\
        \shortintertext{By Theorem \ref{t:mainVar}, $ \Var(\eta_n^N(f)) \leq\blu{\mathbb{E}[(\eta_n^N(f) - \mu_n(f))^2] \leq} 
        \delta \epsilon^2$, giving}
        1 - \delta &\leq P\pa{\ab{\frac{1}{N}\sum_{i=1}^N f(\xi_i^{(n)}) - \mu_n(f)} \leq \epsilon}.
    \end{align*}
\end{proof}
Our results are similar to \cite{mathews2022finite}, whose mixing times have a similar dependence on the spectral gap of the local Markov chain. However, our second assumption on the generators of the 
Markov process allows for us to obtain local mixing time results without restricting the movement of the Markov process to specified partitions of the target measure.  Additionally, the number of samples required in \cite{mathews2022finite} has a dependence on the measure over the specified partition. This is similar to the dependence on $\sum_i\frac{1}{w_i}$ in our results, which worsens for mixtures containing components with low weights. 
We additionally give the following corollary which bounds TV-distance between the empirical distribution and its target. 
\begin{corollary}[\bf Total Variation bound]
\label{t:main} 
Suppose Assumption~\ref{a} holds. Let $\{\xi_n^{(1)}, \dots ,\xi_n^{(N)}\} \sim \hat{\mu}_n^N$ be the joint measure and $\hat{\mu}_n$ be the marginal of the particles produced by SMC with target measure $\mu_n$. \blu{Then for all $\epsilon > 0$, \blu{$M = \max_k M_k$}, \blu{$w^* = \min_{k,i} w_{k}^{(i)}$}, choosing}
\blu{\begin{enumerate}
        \item $ N \geq n\cdot\max\bigg\{\frac{16\gamma M}{w^*\epsilon^2},  \; \frac{128\gamma^{\frac{35}{8}}M^\frac{7}{4}}{(w^*)^\frac{15}{8}}\bigg\}$
        \item $t_k \geq 2C_k^*\gamma^7$
        \end{enumerate}}
    yields 
    \begin{equation*}
       \Vert \hat{\mu}_n - \mu_n \Vert _{\textup{TV}} \leq \epsilon.
    \end{equation*}
\end{corollary}
\begin{proof}
     Using the definition of total variation distance and Cauchy-Schwarz,
     \begin{align*}
         \norm{\hat{\mu}_n - \mu_n}_{\textup{TV}} &= \sup_A | \hat{\mu}_n(I_A) - \mu_n(I_A)|= \sup_A | \mathbb{E}_{\xi_n}[\eta_n^N(I_A)] - \mu_n(I_A)|\\
         &\leq \sup_A \mathbb{E}_{\xi_n}|\eta_n^N(I_A) - \mu_n(I_A)|\leq \sup_A \mathbb{E}_{\xi_n}[|\eta_n^N(I_A) - \mu_n(I_A)|^2]^\frac{1}{2}
     \end{align*}
     We apply Theorem \ref{t:mainVar} with $f = I_A$, noting that $\Vert I_A - \mu_n(I_A)\Vert _{\sup}^2 \leq 1$. This yields $\Var(\eta_n^N(I_A))\le \epsilon^2$ and $\norm{\hat{\mu}_n - \mu_n}_{\textup{TV}} \le \epsilon$. 
 \end{proof}

\blu{Our results require a uniform bound on density ratios as well as on $f-\mu_n(f)$. We leave open the possibility of relaxing these assumptions.}

\section{Applications} \label{s:applications}
In this section we show that Assumption \ref{a} 
is satisfied by a sequence of commonly occurring measures and frequently used Markov kernels, and hence Theorem \ref{t:mainVar} can be applied.
More specifically, we show that condition 3 is satisfied by Langevin diffusion, Glauber Dynamics and Metropolis-Hastings chains (e.g., Metropolis random walk). We show that conditions 1, 2, and 4 are satisfied for sequences of measures defined by power tempering for Gaussian mixtures and Gaussian convolution of arbitrary mixtures. 

Functional inequalities including log-Sobolev inequalities for Glauber dynamics on various distributions have been the subject of intense study; see e.g., \cite{anari2022entropic,blanca2022mixing,chen2022localization}. 
For the Metropolis random walk, weaker inequalities, e.g., for $s$-conductance for log-concave distributions, are known \cite{dwivedi2019log}. 
It should also be noted that our results are for the continuous version of Glauber dynamics and Metropolis-Hastings chains. This is not a problem for an algorithmic result, as the continuous chain can be simulated by drawing transition times according to a Poisson process.

\subsection{Markov Process Decompositions} \blu{In this subsection we show that Assumption \ref{a} (3) is satisfied for Langevin diffusion, Glauber dynamics and Metropolis-Hastings.
Specifically, our bounds require that the Dirichlet form of the continuous process decomposes as $\langle f, \mathscr{L}_kf \rangle_{\mu_k} \leq \sum_{i=1}^{\blu{M_k}} \blu{w_k^{(i)}}\langle f, \mathscr{L}_{ki}f \rangle_{\mu_k^{(i)}}$. 
Later in Section \ref{s:applications} we connect our bounds back to the discretized process.} 
\subsubsection{Langevin Diffusion}\label{langevin diffusion}
We first show that condition (2) of Assumption \ref{a} is satisfied by 
Langevin diffusion, which is a Markov process with stationary distribution $p(x) \propto e^{-V(x)}$. 
Langevin diffusion is described by the stochastic differential equation
\begin{equation}
    dX_t = -\grad V(X_t)\,dt + \sqrt{2}\,dW_t,
\end{equation}
where $W_t$ is Brownian motion. The
time-discretized process with step size $h$ is given by 
\begin{equation}\label{LD}
    X_{t+1} = X_t - h\grad f(X_t) + \sqrt{2h}\cdot \xi_t,  \; \; \; \; \xi_t \sim N(0,I). 
\end{equation}
This can be thought of as a random walk where, at each time step, there is a drift in the direction of the gradient.
\begin{proposition}\label{LD:assumptions}
    For each $k$, let $P_k = (\Omega, \mathscr{L}_k)$ be Langevin diffusion with stationary measure $\mu_k= \sum \blu{w_k^{(i)}}\mu_{k}^{(i)}$. Then condition (2) of Assumptions \ref{a} holds with equality. Namely, 
    \begin{itemize}
        \item[2.] There exists a decomposition 
$$\langle f, \mathscr{L}_kf \rangle_{\mu_k} = \sum_{i=1}^{\blu{M_k}} \blu{w_k^{(i)}}\langle f, \mathscr{L}_{ki}f \rangle_{\mu_k^{(i)}}$$
where $\mathscr{L}_{ki}$ is the generator of 
Langevin diffusion $P_{ki}$ with stationary distribution $\mu_k^{(i)}$. 
    \end{itemize}
\end{proposition}
\begin{proof}
Given Langevin diffusion $\mathscr{L}_k$ with stationary measure $\mu_k$ and density $p_k(x) \propto e^{-g_k(x)}$, we consider the mixture measure $\mu_k = \sum \blu{w_k^{(i)}}\mu_{k}^{(i)}$,  where the density of $\mu_{k}^{(i)}$ is given by 
\begin{equation*}
    p_{ki}(x) \propto \exp(-g_{ki}(x)).
\end{equation*}
For Langevin diffusion for a measure $\mu$ with density $\propto e^{-g}$, we have that $\mathscr{L} = \Delta f - \langle \grad g, \grad f\rangle$. Moreover, by integration by parts we can express the generator as $\langle f , \mathscr{L} f \rangle_\mu = \int_\Omega \norm{\grad f}^2 \mu(dx)$. Applying this definition we can show that condition (2) holds: 
\begin{equation*}
  \langle f, \mathscr{L}_k f\rangle_{\mu_k} = -\int_\Omega \norm{\grad f}^2 p_k(x)dx = -\int_\Omega \norm{\grad f}^2 \sum_{i=1}^{\blu{M_k}} \blu{w_k^{(i)}} p_{ki}(x)dx = \sum_{i=1}^{\blu{M_k}} \blu{w_k^{(i)}} \langle f, \mathscr{L}_{ki}f\rangle_{\mu_k^{(i)}}.
\end{equation*}
\end{proof}
\subsubsection{Glauber dynamics}\label{s:Glauber_dynamics}
We want to show that the decomposition in Assumptions \ref{a} is satisfied for Glauber dynamics for a mixture distribution $p(x) = \sum_{k=1}^M w_k p_k(x)$ on the hypercube $\Omega = \{0,1\}^d.$ Glauber dynamics is the Markov chain whose transition kernel is given as follows: given $x\in \Omega$, 
\begin{enumerate}
    \item Pick $i \in [1,d]$, with probability $\frac{1}{d}$. 
    \item Flip coordinate $x_i$, with probability $p(X_i=1-x_i | X_{-i}=x_{-i}).$
\end{enumerate}
Here, $x_{-i}$ denotes $(x_j)_{j\ne i}\in \{0,1\}^{[n]\backslash \{i\}}$. 
The above process yields the transition probabilities 
$$P(y \vert x) = \begin{cases} 
      \frac{1}{d}\sum_{i=1}^d p(X_i=x_i \vert X_{-i} = x_{-i}) & x=y\\
     \frac{1}{d}p(X_{-i} = 1-x_i\vert X_{-i} = x_{-i}) & \|x-y\|=1, x_i \neq y_i\\
      0 & \|x-y\| > 1 
   \end{cases}
$$
where $\|\cdot\|$ denotes the 1-norm (Hamming distance). The continuous-time Glauber dynamics is the continuous-time Markov process whose generator is given by $\mathscr L = P-I$. 
\begin{proposition}\label{glauber:assumptions}
    Let $P$ be the transition kernel and $\mathscr{L} = P - I$ be the  generator of Glauber dynamics on $\{0,1\}^d$ with stationary measure $p(x) = \sum_{k=1}^m w_kp_{k}(x)$. Then condition (2) of Assumptions \ref{a} holds with equality. Namely, 
    \begin{itemize}
        \item[2.] There exists a decomposition 
$$\langle f, \mathscr{L}f \rangle_{p} \leq \sum_{k=1}^m w_k\langle f, \mathscr{L}_{k}f \rangle_{p_k}$$
where $P_k$ is the transition kernel and $\mathscr{L}_{k} = P_k - I$ is the generator of 
Glauber dynamics with stationary distribution $p_k(x)$. 
    \end{itemize}
\end{proposition}
\begin{proof}
Let $\flipi xi$ denote $x$ with the $i$th coordinate flipped. 
We consider the Dirichlet form for a reversible Markov chain, 
\begin{align*}
    \mathscr{E}_P(f,f) \vcentcolon &= \frac{1}{2}\sum_{x,y\in \Omega}\bigg(f(x) - f(y)\bigg)^2p(x)P(y\vert x)\\
     &=\frac{1}{2}\sum_{\substack{\|x-y\|= 1 \\ x_i \neq y_i \\ 1 \leq i \leq d}}\bigg(f(x) - f(y)\bigg)^2p(x)  \frac{1}{d}P(X_{-i} = y_i\vert X_{-i} = x_{-i})\\
    &= \frac{1}{2}\frac{1}{d}\sum_{i=1}^d\bigg(f(x) - f(\flipi xi)\bigg)^2\sum_kw_kp_k(x)\frac{\sum_k w_k p_k(\flipi{x}{i})}{\sum_k w_k \big( p_k(\flipi{x}{i}) + p_k(x)\big)}\\
    &=\frac{1}{2}\frac{1}{d}\sum_{i=1}^d \bigg(f(x) - f(\flipi xi)\bigg)^2\bigg(\frac{\sum_k w_k p_k(x)}{\sum w_k \big( p_k(\flipi{x}{i}) + p_k(x)\big)}\frac{\sum_k w_k p_k(\flipi{x}{i})}{\sum_k w_k \big( p_k(\flipi{x}{i}) + p_k(x)\big)}\bigg)\\
    &\quad\cdot\sum_k w_k \big( p_k(\flipi{x}{i}) + p_k(x)\big)\\
    \shortintertext{Defining sequences $a_k = \frac{w_k p_k(x)}{w_k\big( p_k(\flipi{x}{i}) + p_k(x)\big)}$, $b_k = \frac{w_k p_k(\flipi{x}{i})}{w_k\big( p_k(\flipi{x}{i}) + p_k(x)\big)}$ with weights $\alpha_k = \frac{w_k \big( p_k(\flipi{x}{i}) + p_k(x)\big)}{\sum_k w_k \big( p_k(\flipi{x}{i}) + p_k(x)\big)}$, noting that $a_k$, $b_k$ are oppositely sorted sequences, we can apply the weighted version of Chebyshev's sum inequality to get, }
    &\geq \frac{1}{2}\frac{1}{d}\sum_{i=1}^d \bigg(f(x) - f(\flipi xi)\bigg)^2\sum_k \frac{w_k \big( p_k(\flipi{x}{i}) + p_k(x)\big)}{\sum_k w_k \big( p_k(\flipi{x}{i}) + p_k(x)\big)}\cdot\frac{p_k(\flipi{x}{i})}{p_k(x)+p_k(\flipi{x}{i})}\\
    &\quad\cdot\frac{p_k(x)}{p_k(x)+p_k(\flipi{x}{i})}\cdot \sum_k w_k \big( p_k(\flipi{x}{i}) + p_k(x)\big)\\
    &=\frac{1}{2}\frac{1}{d}\sum_{i=1}^d \bigg(f(x) - f(\flipi xi)\bigg)^2\sum_k \frac{p_k(\flipi{x}{i})}{p_k(x)+p_k(\flipi{x}{i})}\cdot w_kp_k(x)\\
    &= \sum_k w_k \cdot \frac{1}{2}\sum_{i=1}^d \bigg(f(x) - f(\flipi xi)\bigg)^2p_k(x)P_k(X_{i} = 1-x_{i}\vert X_{-i} = x_{-i})\\
    &= \sum_k w_k \mathscr{E}_{P_k}(f,f).
\end{align*}
By \eqref{dirichlet:discrete} we have that $\langle f, \mathscr{L}f \rangle_{p} \leq \sum_{k=1}^m w_k\langle f, \mathscr{L}_{k}f \rangle_{p_k}$ as desired. 
\end{proof}
\subsubsection{Metropolis-Hastings chain}\label{s:Metropolis_chain} We will show that condition (2) of Assumption \ref{a} is met by Markov chains of Metropolis-Hastings type. A Metropolis-Hastings Markov chain is a discrete-time Markov process that converge to a stationary distribution $\pi$ by iterating two steps. First, given an initial point $x$, a step $y$ is proposed according to the conditional distribution $P(y|x)$ (called the proposal distribution). The proposal step $y$ is accepted with probability
$$ \alpha(x,y) \vcentcolon= \min\bigg(1, \frac{\pi(y)P(x|y)}{\pi(x)P(y|x)}\bigg).$$
The transition kernel $Q(y|x)$ for $y\ne x$ is given by
$$Q(y|x) \vcentcolon= P(y|x)\alpha(x,y).$$
For example, on $\mathbb{R}^d$, $P(y|x)$ is often taken to be a Gaussian distribution centered at $x$, or uniform over a ball or ellipsoid (and the chain is called the Metropolis random walk).
Again, this can be converted to a continuous-time process which can be simulated algorithmically.
\begin{proposition}\label{MH:assumptions}
    Let $Q$ be the transition kernel and $\mathscr{L} = Q-I$ be the generator of the Metropolis-Hastings chain on $\Omega$ with proposal $P$ and stationary measure $\pi(x) = \sum_k w_k\pi_{k}(x)$. Then condition (2) of Assumptions \ref{a} holds. Namely, 
    \begin{itemize}
        \item[2.] There exists a decomposition 
$$\langle f, \mathscr{L}f \rangle_{\pi} \leq \sum_{k=1}^m w_k\langle f, \mathscr{L}_{k}f \rangle_{\pi_k}$$
where $Q_k$ is the transition kernel and $\mathscr{L}_{k} = Q_k - I$ is the generator of the Metropolis-Hastings chain on $\Omega$ with proposal $P$ and stationary distribution $\pi_k(x)$. 
    \end{itemize}
\end{proposition}
\begin{proof}
    For a reversible Markov chain we have that, 
    \begin{align*}
          \mathscr{E}_Q(f,f) \vcentcolon &= \frac{1}{2}\sum_{x,y\in \Omega}\bigg(f(x) - f(y)\bigg)^2\pi(x)Q(y|x)\\
          &= \frac{1}{2}\sum_{x,y\in \Omega}\bigg(f(x) - f(y)\bigg)^2\pi(x)\min\bigg\{1, \frac{\pi(y)P(x|y)}{\pi(x)P(y|x)}\bigg\}P(y|x)\\
          &= \frac{1}{2}\sum_{x,y\in \Omega}\bigg(f(x) - f(y)\bigg)^2\min\bigg\{\sum_kw_k\pi_k(x)P(y|x), \sum_kw_k\pi_k(y)P(x|y)\bigg\}\\
          &\geq \frac{1}{2}\sum_{x,y\in \Omega}\bigg(f(x) - f(y)\bigg)^2\sum_kw_k \min\bigg\{\pi_k(x)P(y|x), \pi_k(y)P(x|y)\bigg\}\\
          &= \sum_kw_k\frac{1}{2}\sum_{x,y\in \Omega}\bigg(f(x) - f(y)\bigg)^2\min\bigg\{1, \frac{\pi_k(y)P(x|y)}{\pi_k(x)P(y|x)}\bigg\}\pi_k(x)P(y|x)\\
          &=\sum_k w_k \mathscr{E}_{Q_k}(f,f)
    \end{align*}
    By definition \ref{dirichlet:discrete} we have that $\langle f, \mathscr{L} f \rangle_\pi \leq \sum_k w_k \langle f , \mathscr{L}_k f \rangle_{\pi_k}.$
\end{proof}
\subsection{Target Distributions and Measure Sequences} In this subsection we show that our main Theorem \ref{t:main} is applicable to some representative sequences of multimodal distributions. 
We demonstrate that two standard methods for constructing the sequence of measures, powering tempering and Gaussian convolution, satisfy Assumptions \ref{a}.

In each setting we show that its possible to finitely bound the density ratio $\gamma$ with polynomial dependence on the parameters. 
Additionally, we show that in the case of Langevin dynamics, in both settings, it is possible to bound the local LSI constant $C_k^*$. 

\subsubsection{Gaussian Mixtures and Power Tempering} \label{pt}
\blu{We now apply our results to a concrete setting where the target distribution $\pi(x) = \sum_k \alpha_k \cdot q(x; \mu_k, \Sigma_k)$ is a mixture of Gaussians (with means $\mu_k$ and variances $\Sigma_k$). This verifies that SMC is efficient in a setting where simulated tempering is known to succeed \cite{ge2020simulated}, with a different proof. We will consider the sequence of measures $\pi_1, \dots, \pi_n$ defined by power tempering $\pi_i(x) \propto \pi(x)^{\beta_i}$ for a chosen temperature ladder $\beta_1, \dots, \beta_n$. Power tempering is a commonly used and well studied method \cite{Marinari_1992,neal2001annealed,ge2020simulated} for defining a sequence of distributions and in this setting, the rest of our Assumptions \ref{a} can be easily verified.}

\begin{proposition}
    Let the target distribution be a mixture of $M$ Gaussians 
    $$\pi(x) = \sum_{k=1}^M \alpha_kq(x; \mu_k, \Sigma_k), \qquad \sum_{k=1}^M \alpha_k = 1,\qquad \alpha_k> 0,$$
    where $q(x; \mu_k, \Sigma_k)$ denotes the density of $\mathcal{N}(\mu_k, \Sigma_k).$ Let $0 < \beta_1 < \dots < \beta_n = 1$ be a temperature ladder and define the sequence of measures
    $$\pi_i(x) \propto \tilde{\pi}_i(x) \vcentcolon= \pi(x)^{\beta_i}.$$
    Then the following properties hold: 
    \begin{enumerate}
        \item (Density Ratio Bound) Assumption \ref{a} (2) is satisfied with $\frac{d\pi_i}{d\pi_{i-1}} \leq \gamma_i$, where for $\Delta \beta = \beta_i - \beta_{i-1}$,
        $$\gamma_i = \frac{1}{\min_k \alpha_k }\left(\frac{\beta_i}{\beta_{i-1}} \right)^{d/2}\left(\frac{\max_k|\Sigma_k|}{\min_k|\Sigma_k|}\right)^{\Delta\beta/2}.$$
        \item (Local LSI Constants) Each $\pi_i$ is expressible as a mixture $\pi_i = \sum w_k^{(i)}\nu_k^{(i)}$, where the components $\nu_k^{(i)}$ are defined by 
        $$\nu^{(i)}_k(x) \propto \frac{\pi(x)^{\beta_i}}{\sum_{j=1}^M\alpha_j q(x; \mu_j, \Sigma_j)^{\beta_i}}q(x;\mu_k,\Sigma_k)^{\beta_i}.$$
        Then Assumption \ref{a} (4) holds for each $\nu_k^{(i)}$ with
        $$C_{\textup{LSI}}(\nu_k^{(i)}) \leq \frac{1}{\min_j \alpha_j}\cdot \frac{\lambda_{\max}(\Sigma_k)}{\beta_i}.$$
        \item (Weight Lower Bound) The normalized mixture weights satisfy 
        $$\min_{i,k} w_{k}^{(i)} \geq \min_{i,k}\frac{\int q(x;\mu_k,\Sigma_k)^{\beta_i}dx}{\sum_j\alpha_j\int q(x;\mu_k,\Sigma_j)^{\beta_i}dx}\min_k\alpha_k^2 .$$
       If $\Sigma_j = \Sigma_k$, for all $j,k \in [1,M]$, then this reduces to $\min_{i,k} w_{k}^{(i)} \geq \min_k\alpha_k^2$.
    \end{enumerate}
\end{proposition}
\begin{proof}
Let $Z_k = \int \pi(x)^{\beta_k} dx$. To bound the ratio we consider 
$$\frac{\pi_i(x)}{\pi_{i-1}(x)} = \frac{Z_{i-1}}{Z_i}\frac{\pi(x)^{\beta_i}}{\pi(x)^{\beta_{i-1}}} = \frac{Z_{i-1}}{Z_i} \pi(x)^{\Delta\beta}.$$
Then the density $\pi(x)$ can be bounded by 
\begin{align}
\label{e:gm-eq-1}
\pi(x) &= \sum_{k=1}^M \alpha_k q_k(x) \leq \max_k \left\| q_k\right\|_\infty = \max_k \frac{1}{(2\pi)^{d/2}|\Sigma_k|^{1/2}}.
\end{align}
To bound the normalizing constants, using that $\beta_i \in (0,1)$, we have by the power mean inequality that
$$Z_i = \int \left(\sum_{k=1}^M \alpha_k q_k(x) \right)^{\beta_i}dx \geq \sum_{k=1}^M \alpha_k \int q_k(x)^{\beta_i} dx.$$
We find an upper bound by 
$$\pi(x)^{\beta_i} \leq \sum_k\alpha_k^{\beta_i}q_k(x)^{\beta_i} = \sum_k \frac{\alpha_k}{\alpha_k^{1-\beta_i}}q_k(x)^{\beta_i} \leq \frac{1}{\min_k \alpha_k }\sum_k \alpha_kq_k(x)^{\beta_i}.$$
Then 
$$Z_i = \int\pi(x)^{\beta_i} dx\leq \frac{1}{\min_k \alpha_k }\sum_{k=1}^M \alpha_k\int q_k(x)^{\beta_i}dx.$$
Therefore, using that $\int q_k(x)^{\beta_i} dx = \frac{(2\pi)^{d(1-\beta_i)/2}}{\beta_i^{d/2}}|\Sigma_k|^{(1-\beta_i)/2}$ the ratio $\frac{Z_{i-1}}{Z_i}$ is bounded by
\begin{align}
\nonumber
\frac{Z_{i-1}}{Z_i} &\leq \frac{1}{\min_k \alpha_k }\frac{\sum_{k=1}^M \alpha_k\int q_k(x)^{\beta_{i-1}}dx}{\sum_{k=1}^M \alpha_k \int q_k(x)^{\beta_i}dx} \leq \frac{1}{\min_k \alpha_k }\cdot\max_k\frac{\int q_k(x)^{\beta_{i-1}}dx}{\int q_k(x)^{\beta_{i}}dx} \\
&\leq \frac{1}{\min_k \alpha_k }\left(\frac{\beta_i}{\beta_{i-1}} \right)^{d/2}(2\pi)^{d\Delta\beta/2}\max_k|\Sigma_k|^{\Delta\beta/2}.
\label{e:gm-eq-2}
\end{align}
Combining \eqref{e:gm-eq-1} and \eqref{e:gm-eq-2} yields the first part. 

Next we show an upper bound on the LSI constant. Consider the distribution ${\pi}_{\textup{mix}, i}(x) \propto \sum_{k=1}^M \alpha_k q(x; \mu_k, \Sigma_k)^{\beta_i}=: \tilde{\pi}_{\textup{mix},i}$. It's shown in  \cite[Lemma 7.3]{ge2020simulated} that 
 \begin{equation}\label{eq: lem 7.3}
 \tilde{\pi}_{\textup{mix}, i}(x) \leq \tilde{\pi}_i(x) \leq \frac{1}{\min_k \alpha_k}\tilde{\pi}_{\textup{mix}, i}(x). 
 \end{equation}
 Next, we will consider a hypothetical mixture $\tilde{\pi}_i(x) = \sum_k \alpha_{k}^{(i)} \tilde{\nu}_{k}^{(i)}(x)$ for some $\alpha_k^{(i)}$, $\tilde{\nu}_k^{(i)}$. 
 We choose $\alpha_k^{(i)} = \alpha_k$ and $\tilde{\nu}_{k}^{(i)}=\frac{\tilde{\pi}_i(x)}{\tilde{\pi}_{\textup{mix}, i}(x)}q(x;\mu_k,\Sigma_k)^{\beta_i}$. 
 Then considering $\tilde{\nu}_k^{(i)}$, by (\ref{eq: lem 7.3}), 
 \begin{align*}
     q(x;\mu_k,\Sigma_k)^{\beta_i} \leq \frac{\tilde{\pi}_i(x)}{\tilde{\pi}_{\textup{mix}, i}(x)}q(x;\mu_k,\Sigma_k)^{\beta_i}\leq \frac{1}{\min_k \alpha_k}q(x;\mu_k,\Sigma_k)^{\beta_i}.
 \end{align*}
 Then by Holley-Stroock (\cite[Proposition 2.3.1]{chewi2023}) we have that 
 $$C_{\textup{LSI}}(\nu_k^{(i)}) \leq \frac{1}{\min_k \alpha_k}C_{\textup{LSI}}(q(x;\mu_k,\Sigma_k)^{\beta_i}),$$
 where we use $C_{\textup{LSI}}(q)$ to denote the log-Sobolev constant of the normalized density $q/\int q$. 
 Since $q(x; \mu_k, \Sigma_k)^{\beta_i} = C \cdot q(x; \mu_k, \frac{\Sigma_k}{\beta_i})$ for some constant $C$, the Bakry-Emery criterion says that $C_{\textup{LSI}}(q(x;\mu_k,\Sigma_k)^{\beta_i}) \leq \frac{\lambda_{\max}(\Sigma_k)}{\beta_i}$ hence 
 $$C_{\textup{LSI}}(\nu_k^{(i)}) \leq \frac{1}{\min_k \alpha_k}\cdot \frac{\lambda_{\max}(\Sigma_k)}{\beta_i}.$$
 Lastly, we need to show that given our choice of component functions $\nu_j^{(i)}$, the weights corresponding to the normalized components don't become too small.
We have that 
\begin{align*}
    w_j^{(i)} &= \frac{\alpha_j\int \tilde{\nu}_j^{(i)}(x)dx}{\sum_k\alpha_k\int \tilde{\nu}_k^{(i)}(x)dx} = \frac{\alpha_j\int \frac{\tilde{\pi}_i(x)}{\tilde{\pi}_{\textup{mix}, i}(x)}q(x;\mu_j,\Sigma_j)^{\beta_i}dx}{\sum_k\alpha_k\int \frac{\tilde{\pi}_i(x)}{\tilde{\pi}_{\textup{mix}, i}(x)}q(x;\mu_k,\Sigma_k)^{\beta_i}dx}\\
    \shortintertext{then by (\ref{eq: lem 7.3}),}
    &\geq\frac{\alpha_j\int q(x;\mu_j,\Sigma_j)^{\beta_i}dx}{\frac{1}{\min_k\alpha_k}\sum_k\alpha_k\int q(x;\mu_k,\Sigma_k)^{\beta_i}dx}=\frac{\alpha_j\int q(x;\mu_j,\Sigma_j)^{\beta_i}dx}{\sum_k\alpha_k\int q(x;\mu_k,\Sigma_k)^{\beta_i}dx}\min_k\alpha_k 
\end{align*}
Therefore, in the case of $\Sigma_k = \Sigma_j$ for all $j,k$, it follows that $\min_{i,k} w_k^{(i)} \geq (\min_k \alpha_k)^2$.
\end{proof}
The following then follows directly from Corollary \ref{t:main}. 
\begin{corollary}(Total Variation Bounds for Power Tempering) 
Let $\{\xi_n^{(1)}, \dots ,\xi_n^{(N)}\} \sim \hat{\mu}_n$ be the particles produced by SMC with target distribution $\mu_n$. 
Let the target distribution be a mixture of $M$ Gaussians 
    $$\pi(x) = \sum_{k=1}^M \alpha_kq(x; \mu_k, \Sigma), \qquad \sum_{k=1}^M \alpha_k = 1,\qquad \alpha_k> 0,$$
    where $q(x; \mu_k, \Sigma)$ denotes the density $\mathcal{N}(\mu_k, \Sigma)$. Let $0 < \beta_1 < \dots < \beta_n = 1$ be a temperature ladder and define the sequence of measures $\pi_i(x) \propto \tilde{\pi}_i(x) \vcentcolon= \pi(x)^{\beta_i}.$ Then choosing 
    \begin{enumerate}
        \item $N \geq n \max\left\{\frac{16}{\epsilon^2\min_k \alpha_k^3 }\max_i\left(\frac{\beta_i}{\beta_{i-1}} \right)^\frac{d}{2}
        , \frac{128M^\frac{7}{4}}{\min_k\alpha_k^\frac{65}{8}}\left(\frac{\beta_i}{\beta_{i-1}} \right)^\frac{35d}{16}
        \right\}$
        \item
        $t_k\ge \frac{2C_k^*}{\min_k \alpha_k^7 }\max_i\left(\frac{\beta_i}{\beta_{i-1}} \right)^{\frac{7d}{2}}$
    \end{enumerate}
        yields 
    \begin{equation*}
       \Vert \hat{\mu}_n - \mu_n \Vert _{\textup{TV}} \leq \epsilon.
    \end{equation*}
    Choosing $\fc{\be_i}{\be_{i-1}}=1+O\pf{1}{d}$ makes this polynomial in all parameters.
\end{corollary}
\color{black}
\subsubsection{Gaussian Convolutions and Noise Contrastive Estimation}\label{gc}
Noise Contrastive Estimation (NCE) is a estimation method for parametrized statistical models \cite{JMLR:v13:gutmann12a} which has recently been adapted in machine learning to train energy-based models. Given a set of points $\{x_n^{(i)}\}_{i=1}^N$ drawn from $\mu_n$, NCE produces an estimate of the data distribution $\mu_n$ by comparing these data points to points $\{x_1^{(i)}\}_{i=1}^N$ generated from some known noise distribution $\mu_1$ and estimating $\frac{d\mu_n}{d\mu_1}$. 
When the sampled points $\{x_1^{(i)}\}_{i=1}^N$ and $\{x_n^{(i)}\}_{i=1}^N$ have little overlap, NCE can provide poor estimates of $\mu_n$. Telescoping Ratio Estimation (TRE) is an extension of NCE that circumvents this issue by interpolating between the two sets of points to obtain the sets  $\{x_k^{(i)}\}_{i=1}^N$, $k=2, \dots, n-1$ \cite{rhodes2020telescoping}. NCE is applied between sets  $\{x_k^{(i)}\}_{i=1}^N$ and  $\{x_{k+1}^{(i)}\}_{i=1}^N$ to learn $\frac{d\mu_{k+1}}{d\mu_k}$ and hence $\mu_{k+1}$ inductively, for $k =1, \dots, n$.

\blu{We now apply our results in a concrete setting where our sequence of distributions $\mu_k$ is constructed via Gaussian convolution. This is a common technique in diffusion models and score-based generative modeling and also fits naturally in the NCE/TRE framework—since intermediate point sets can be constructed by adding Gaussian noise.}

\blu{We further specialize to the case where $\mu_n$ is a mixture, $\mu_n = \sum_{i=1}^M w_i\mu_n^{(i)}$. This provides a non-trivial example that satisfies our Assumptions \ref{a}:
\begin{enumerate}
    \item It illustrates a special case of our framework where the mixture weights are constant, as $\mu_k = \mu_n * \varphi_k = \sum_i w_i(\mu_n^{(i)}*\varphi_k)$. 
    While the density ratio $\bar{g}_{k,k+1}$ is not analytically computable (a challenge which we address in remark \ref{r:mu1}), it can be bounded by a uniform constant $\lambda$. This is shown explicitly in Proposition \ref{ratio_example}. 
    \item It provides a concrete setting where Assumptions \ref{a} part 4 (LSI constants) can be explicitly verified. This is shown in a more general setting (Lemma \ref{lem:PI}), but in our special setting, since the convolution of Gaussians is Gaussian, the LSI can be explicitly bounded. 
\end{enumerate}}

The following proposition shows how intermediate point sets can be constructed by adding Gaussian noise. 

\begin{proposition}\label{waymark} Let $\{x_n^{(i)}\}_{i=1}^N$ be from some data distribution with unknown measure $\pi_n$ and define
$$x_{k-1}^{(i)} \vcentcolon = x_{k}^{(i)} + \sigma \xi, \; \xi \sim N(0,1),$$
for $k= 2, \dots, n$. Then the measure $\mu_k$ of $x_k^{(i)}$ is given by 
$$\mu_k = \mu_n*\varphi_k,$$
where $\varphi_k \sim N(0,\sum_{i=k+1}^n\sigma_i^2)$.
\end{proposition}
\begin{proof}
    Follows by additivity of variance for Gaussians.
\end{proof}

As defined in Proposition \ref{waymark}, $\varphi_1$ is a high variance Gaussian, which when convolved with a measure $\mu_n$ yields a measure that is approximately Gaussian itself. Therefore, 
we can approximately sample from $\mu_n*\varphi_1$; see Remark \ref{r:mu1} for details.
We now show that, given a target distribution $\mu_n$, our main result can be applied to the sequence of distributions $\mu_1,\dots,\mu_n$ where $\mu_k$, is defined as in Proposition \ref{waymark}.
\begin{proposition}\label{ratio_example} Let $\mu_n = \sum_i w_i \mu_n^{(i)}$ be the target measure and let 
    \begin{equation*}
        \varphi_k(x) \propto e^{-\beta_k\frac{\norm{x}^2}{2\sigma^2}},
    \end{equation*}
    with $0 <\beta_1< \dots < \beta_{n-1}$.

    Then condition (1) of Theorem \ref{t:main} holds with
    
    $$\frac{\mu_n*\varphi_k(x)}{\mu_n*\varphi_{k-1}(x)} \leq \gamma = \bigg(\frac{\beta_k}{\beta_{k-1}}\bigg)^\frac{d}{2}.$$
\end{proposition}
\begin{proof}
    Consider the following: 
    \begin{align*}
       \frac{\mu_n*\varphi_k(x)}{\mu_n*\varphi_{k-1}(x)} &= \frac{\sqrt{\beta_k^d}\sum_{i=1}^m w_i\int d\mu_n^{(i)}(y) \exp(-\beta_k\frac{\norm{x-y}^2}{2\sigma^2})dy}{\sqrt{\beta_{k-1}^d}\sum_{i=1}^mw_i \int d\mu_n^{(i)}(y) \exp(-\beta_{k-1}\frac{\norm{x-y}^2}{2\sigma^2})dy}\\
        &=\frac{\sqrt{\beta_k^d}\sum_{i=1}^m w_i\int d\mu_n^{(i)}(y) \exp(-(\beta_{k-1}+\beta_k-\beta_{k-1})\frac{\norm{x-y}^2}{2\sigma^2})dy}{\sqrt{\beta_{k-1}^d}\sum_{i=1}^mw_i \int d\mu_n^{(i)}(y) \exp(-\beta_{k-1}\frac{\norm{x-y}^2}{2\sigma^2})dy}\\
    &= \frac{\sqrt{\beta_k^d}\sum_{i=1}^m w_i\int d\mu_n^{(i)}(y) \exp(-\beta_{k-1}\frac{\norm{x-y}^2}{2\sigma^2})\exp(-(\beta_k-\beta_{k-1})\frac{\norm{x-y}^2}{2\sigma^2})dy}{\sqrt{\beta_{k-1}^d}\sum_{i=1}^mw_i \int d\mu_n^{(i)}(y) \exp(-\beta_{k-1}\frac{\norm{x-y}^2}{2\sigma^2})dy}\\
    &\leq \frac{\sqrt{\beta_k^d}\sum_{i=1}^m w_i\int d\mu_n^{(i)}(y) \exp(-\beta_{k-1}\frac{\norm{x-y}^2}{2\sigma^2})dy}{\sqrt{\beta_{k-1}^d}\sum_{i=1}^mw_i \int d\mu_n^{(i)}(y) \exp(-\beta_{k-1}\frac{\norm{x-y}^2}{2\sigma^2})dy}\\
    &= \bigg(\frac{\beta_k}{\beta_{k-1}}\bigg)^\frac{d}{2}.
    \end{align*}
\end{proof}

\subsubsection{TV-bounds for Sampling from Telescoping Ratio Estimation} We combine our results form Section \ref{langevin diffusion} and Section \ref{gc} to find TV-bounds for SMC applied to a sequence of distributions $\mu_1, \dots, \mu_n$ learnt perfectly via TRE. We further show that the dependence on the log Sobolev constants in our TV-bound from Theorem \ref{t:main} only depends on the local log Sobolev constants of the target mixture measure $\mu_n = \sum_i w_i \mu_n^{(i)}$. 

 \begin{lem}[{\cite{chafai2004entropies}, \cite[Proposition 2.3.7]{chewi2023}}] \label{lem:PI} Suppose that $\pi = \pi_1 * \pi_2$ where $\pi_1, \pi_2$ satisfy a Poincar\'e inequality. Then, $\pi$ also satisfies the corresponding functional inequality 
$$C_{\textup{LSI}}(\pi) \leq C_{\textup{LSI}}(\pi_1) + C_{\textup{LSI}}(\pi_2),$$
where $C_{\textup{LSI}}(\pi)$ is the log Sobolev constant of $\pi$.
\end{lem}
Since 
    $$\pi_k = \pi_n*\varphi_k = \sum_i w_i \pi_n^{(i)}*\varphi_k = \sum_i w_i (\pi_n^{(i)}*\varphi_k),$$
it is natural to define $\pi_k$ as the mixture measure 
$$\pi_k = \sum_i w_i (\pi_n^{(i)}*\varphi_k) = \sum_i w_i \pi_k^{(i)},$$
with $\pi_k^{(i)} = \pi_n^{(i)}*\varphi_k$. This allows for us to combine our results from Theorem \ref{t:main} with Lemma \ref{lem:PI} to get a bound strictly in terms of the log-Sobolev constant of the target measure.
\begin{proposition}\label{t:main_example}
    Let $\{\xi_n^{(1)}, \dots ,\xi_n^{(N)}\} \sim \hat{\mu}_n$ be the particles produced by SMC with target distribution $\mu_n$ and at each level of the sequence let 
   $P_k^t$ be the Markov semigroup of Langevin diffusion with generator $\mathscr{L}_k$, stationary distribution $\mu_k$, and Dirchlet form $\mathscr{E}_k(f,f) = -\langle f, \mathscr{L}_kf\rangle_{\mu_k}$. At each level, assume $\mu_k = \mu_n*\varphi_k$, where $\varphi_k(x) \propto e^{-\beta_k\frac{\norm{x}^2}{2\sigma^2}}$. Moreover, assume $\mu_n$ decomposes as the mixture $\mu_n = \sum_iw_i\mu_n^{(i)}$ \blu{with $C^* =\max_i C_{\textup{LSI}}(\mu_n^{(i)})$}. 
    
Then for all $\epsilon > 0$, letting $M = \max_k M_k, \gamma = \bigg(\frac{\beta_k}{\beta_{k-1}}\bigg)^\frac{d}{2}$ and choosing
 \blu{\begin{enumerate}
        \item $ N \geq n\cdot\max\bigg\{\frac{16\gamma M}{w^*\epsilon^2},  \; \frac{128\gamma^{\frac{35}{8}}M^\frac{7}{4}}{\blu{(w^*)}^\frac{15}{8}}\bigg\}$
        \item $t_k \geq 2\pa{C^*+\frac{\sigma^2}{\beta_k}} \gamma^7$
        \end{enumerate}}
    yields
    \begin{equation*}
        \Vert \hat{\mu}_n - \mu_n\Vert _{\textup{TV}} \leq \epsilon.
    \end{equation*}
\end{proposition}

\begin{proof} We show that all assumptions from Theorem \ref{t:main} are met: 
\begin{enumerate}
    \item The given $\gamma$ is such that $\frac{d\mu_k}{d\mu_{k-1}} \leq \gamma = \bigg(\frac{\beta_k}{\beta_{k-1}}\bigg)^\frac{d}{2}$ for all $k$. This holds by Proposition \ref{ratio_example}.
    \item There exists a decomposition 
$$\langle f, \mathscr{L}_kf \rangle_{\mu_k} = \sum_{i=1}^m w_i\langle f, \mathscr{L}_{ki}f \rangle_{\mu_k^{(i)}}$$
where $\mathscr{L}_{ki}$ is the generator of some Markov chain $P_{ki}$ with stationary distribution $\mu_k^{(i)}$. This holds by Proposition \ref{LD:assumptions}.
\item There exists a log Sobolev inequality of the form
$$\Ent{f^2}{\mu_k^{(i)}}\leq \big(C_{\textup{LSI}}(\mu_n^{(i)}) + C_{\textup{LSI}}(\varphi_k)\big)
\cdot\mathscr{E}_{ki}(f,f),$$
where $C_{\textup{LSI}}(\pi)$ is the log Sobolev constant of $\pi$. This holds by noting $\mu_k^{(i)} = \mu_n^{(i)} * \varphi_n$ and 
then applying Lemma \ref{lem:PI} to get that $C_{\textup{LSI}}(\mu_k^{(i)}) \leq C_{\textup{LSI}}(\mu_n^{(i)}) +C_{\textup{LSI}}(\varphi_k)$.
\end{enumerate}
\end{proof}
\begin{rem}\label{r:mu1}
\blu{It should be noted that the focus of our work is on analyzing sampling error. To get complete non-asymptotic error bounds for the TRE problem, one would also need to consider the error from learnt ratios $\hat{g}_{k,k+1}$. Let $\hat{\mu}^{\text{approx}}_n $ be the marginal distribution from the algorithm on the NCE learned ratios $\hat{g}_{k,k+1}$ and $\hat{\mu}_n$ be the marginal of the idealized SMC (using the true ratios $\bar{g}_{k,k+1}$. The total error in TV distance can be given as: 
\begin{align} \label{eq:error_separation_tv}
\|\hat{\mu}_n^{\text{approx}} - \mu_n\|_{\textup{TV}} &= \|(\hat{\mu}_n^{\text{approx}} - \hat{\mu}_n) + (\hat{\mu}_n - \mu_n)\|_{\textup{TV}} \nonumber \\
&\le \underbrace{\|\hat{\mu}_n - \mu_n\|_{\textup{TV}}}_{\text{(I) SMC Sampling Error}} + \underbrace{\|\hat{\mu}_n^{\text{approx}} - \hat{\mu}_n\|_{\textup{TV}}}_{\text{(II) Ratio Approximation Error}}.
\end{align}
}
\blu{This decomposes the total error from the TRE estimate as a two part problem, to which our SMC analysis provides bounds for (I).} 
\end{rem}
\subsection{Implementation: Initialization \& Discretization}

The results in Subsections \ref{pt} and \ref{gc} are for the idealized Langevin diffusion.
    To obtain an end-to-end result algorithmic result, it remains to argue that (1) we can efficiently draw approximate samples from (a) $\mu_1 = \left(\sum_{k=1}^M \alpha_kq(x; \mu_k, \Sigma_k)\right)^{\beta_1}$ (the warm distribution obtained by power tempering) and (b) $\mu_1  = \mu_n * \varphi_1$ (the distribution obtained by convolving $\mu_n$ with a Gaussian of high variance $\frac{\sigma}{\sqrt{\beta_1}}$); 
    and (2) we can bound the error from discretization of Langevin dynamics. Obtaining guarantees in TV distance for both steps, we can then use a standard coupling argument to obtain guarantees for approximate SMC samples.

    For (1a), note that a (global) Poincar\'e inequality holds for $\mu_1$ when  $\beta_1 = O\pf{\sigma^2}{d^2}$ by \cite{ge2020simulated}, so that it can be approximately sampled efficiently using Langevin. 
    (Specifically, Theorem 7.1 therein shows a Poincar\'e inequality for $\pi_{\textup{mix},1}$, and then comparison using Holley-Stroock with Lemma 7.3 therein give the Poincar\'e inequality for $\pi_1$.)

    For (1b), we claim that for $\be_1$ small enough, it suffices to draw samples from the Gaussian with variance $\fc{\sigma^2}{\be_1}$. 
    By Pinsker's inequality and the regularization effect of Gaussian convolution \cite[Lemma 4.2]{bobkov2001hypercontractivity},
    \[
    \TV(N(0,\sigma^*), \mu_n*\varphi_1) \le 
    \sqrt{\rc 2\KL(N(0,\sigma^*)\;\Vert  \; \mu_n*\varphi_1)} \leq \frac{W_2(N(0, \sigma^* - \frac{\sigma^2}{\beta_1}),\mu_n)}{\frac{\sigma}{\sqrt{\beta_1}}}
    = \frac{\mathbb{E}_{\mu_n}[\norm{X}^2]}{\frac{\sigma}{\sqrt{\beta_1}}}
    ,
    \]
where $W_2$ is the Wasserstein 2-distance. In the last equality we are choosing $\sigma^* = \frac{\sigma^2}{\beta_1}$. The TV distance can be made small by choosing $\be_1$ small enough. Alternatively, we can use rejection sampling with a Gaussian as a proposal distribution to obtain exact samples.

For (2) by letting $\mu_T$ be the measure of the discretized Langevin Monte Carlo process with a fixed step size $h$ and $\pi_T$ the measure of the continuous Langevin dynamics with stationary distribution $\pi$ we have, 
$$\Vert  \mu_T - \pi\Vert _{\textup{TV}} \leq \Vert  \mu_T - \pi_T\Vert _{\textup{TV}} + \Vert  \pi_T - \pi\Vert _{\textup{TV}}$$
where bounds for $\Vert \mu_T - \pi_T\Vert _{\textup{TV}}$ 
under smoothness of $\pi$ are given by \cite[Lemma 2]{dalalyan2017theoretical}.

\section{Overview of Proof} \label{s:Overview}

In this section we first provide a conceptual overview of our main results and then follow this with a more detailed analytical summary.

\blu{Conceptually, each resampling step introduces error, which is alleviated by mixing of the Markov process. In their recursive framework to bound the error, \cite{schweizer2012nonasymptotic} quantify the mixing through variance decay and hypercontractivity for the Markov process. 
Specifically, they use the fact that under global functional inequalities (log-Sobolev constant $C$), the following hold: 
\begin{enumerate}
    \item (Variance decay) $\Var_\pi(P_tf) \leq e^{-\frac{2t}{C}}\Var_\pi(f)$.
     \item (Hypercontractivity) $\ve{P_tf}_{q(t)} \leq \ve{f}_p$, where $q(t) =1+(p-1)e^\frac{2t}{C}$. 
\end{enumerate}
In our case, global functional inequalities are not available. 
We prove analogous bounds for mixture distributions under local functional inequalities, (that is, $\pi = \sum_k w_k \pi_k$ where each $\pi_k$ has satisfies a log-Sobolev inequality with constant $C^*$):
    \begin{enumerate}
        \item (Intra-mode variance decay) $\Var_\pi(P_tf) \leq \frac{C^*}{2t}\Var_\pi(f) +  \sum_k\frac{1}{w_k} \pi(f)^2$.
        \item (Hypercontractivity for mixture) $\frac{1}{(w^*)^{1/q(t)}}\ve{P_tf}_{L_{q(t)}(\pi)} \leq \frac{1}{(w^*)^{1/p}}\ve{f}_{L_p(\pi)}$, where $q(t) = 1+(p-1)e^\frac{2t}{C^*}$.
    \end{enumerate}    
}

Analytically, in Section \ref{Var bounds}, we adapt the framework of \cite{schweizer} used to obtain variance bounds for the SMC estimate.

\begin{enumerate}
    \item We leverage that $\Var(\mu_n^N(f)) = \mathbb{E}[(\mu_n^N(f)- \mu(f))^2]$ can be bounded by $\Var(\nu_n^N(f)) = \mathbb{E}[(\nu^
N_n(f) - \mu(f))^2]$, where $\nu_n^N(f)$ is an unbiased estimate of $\mu_n^N(f)$ (Lemma \ref{l:2.2}). 
    \item 
    We show that $\Var(\nu_n^N(f))$ can be bounded in terms of two quantities $\hat c_n$ and $\bar v_n$. These quantities can be bounded explicitly given two ingredients:
    \begin{enumerate}
        \item 
        A single-step bound of the form
        \begin{align}\Vert \hat{q}_{k-1,k}(f)\Vert _{L_2(\mu_{k-1})}^2 &\leq \alpha\Vert f\Vert _{L_2(\mu_{k})}^2 +  \beta\mu_k(f)^2.
        \label{e:single-step}
        \end{align}
        for all $1<k\leq n$ and some $\alpha<1$. 
        By making effective use of global mixing, \cite{schweizer} assumes $\beta < 1$. We modify these results, allowing for $\beta \ge 1$.  
        \item A hypercontractive inequality of the form
        \begin{align*}
            \norm{P_{k} f}_{L_{p}(\mu_k)}
            &\le \theta(p,p/2) \norm{f}_{L_{p/2}(\mu_k)}
        \end{align*}
        for all $k$.
    \end{enumerate}
\end{enumerate}
In context, $\alpha$ controls the variance bound on a single SMC step \eqref{e:single-step}. Most notably, $\alpha < 1$ allows us to iterate over $n$ steps of SMC without accruing error.  To obtain 
a bound \eqref{e:single-step} using only local mixing we decompose the variance into the intra-mode and inter-mode variances of the mixture measures $\mu_k = \sum_i w_i \mu_k^{(i)}$ (see \eqref{e:inter-intra}). By utilizing this variance decomposition we are able to bound the intra-mode variance using local Poincar\'e constants. 
\begin{enumerate}
    \item In Section \ref{local mixing}, we obtain a bound for the intra-mode variance which depends only on local mixing (Lemma \ref{intra}). Similarly, we obtain a bound for the inter-mode variance. These bounds depend on the local Poincar\'e constants and the minimum mode weight $w^* = \min_k w_k$.
    \item In Section \ref{mixing times}, we obtain a bound of the form $ \Vert \hat{q}_{k-1,k}(f)\Vert _{L_2(\mu_{k-1})}^2 \leq \alpha_k\Vert f\Vert _{L_2(\mu_{k})}^2 +  \beta\mu_k(f)^2$ (Lemma \ref{l:af+b}), where $\alpha_k$ depends on our local mixing results from Section \ref{local mixing}. By bounding $\alpha_k$ for all $k$ we are able to obtain the bound in equation \eqref{e:single-step}) above.
    \item In Section \ref{s:hypercontractivity}, we obtain a bound of the form  $\norm{P_{k} f}_{L_{p}(\mu_k)}
            \le \theta(p,p/2) \norm{f}_{L_{p/2}(\mu_k)}$. We are able to determine $\theta(p,p/2)$ as a function of $p$ and  the minimum mode weight $w^* = \min_k w_k$. This provides us with a closed form for $\hat{c}_n$ in our final results.  
    \end{enumerate}
Finally, to obtain better mixing time results we show that for any $\alpha <1$ such that $\alpha_k < \alpha$ for all $k$, Theorem \ref{t:main} holds. This yields mixing times which depend only on the local mixing given by each $\alpha_k$.

\section{Non-Asymptotic Variance Bounds} \label{Var bounds}
First we note that the error of the SMC estimate can be bounded by the variance of $\nu_n^N(f)$. 
\begin{lem}[{\cite[Lemma 2.2]{schweizer}}]  \label{l:2.2} For $f \in B(E)$ we have the bound 
    \begin{equation*}
        \mathbb{E}[(\eta_n^N(f) - \mu_n(f))^2] \leq 2\Var(\nu_n^N(f)) + 2\Vert f - \mu_n(f)\Vert ^2_{\sup,n}\Var(\nu_n^N(1)).
    \end{equation*}
    \end{lem}
We now state and apply the main theorem of Schweizer, which provides non-asymptotic bounds for the variance of $\nu_n^N(f)$. 
The bound depends on quantities $\bar v_n$ and $\hat c_n$ which depend on how the variance and norms increases through the levels. 
By modifying (with proof) several of the lemmas in Schweizer, we will show that obtaining bounds on these quantities reduce to obtaining a one-step bound and hypercontractivity of the kernels.

\begin{thm}[{\cite[Theorem 2.1]{schweizer}}] 
\label{t:2.1}
Let $N \geq 2\bar{c}_n$. Then for any $f\in B(E)$, we have 
\begin{align}
    \nonumber
    \Var(\nu_n^N(f)) = \mathbb{E}[|\nu_n^N(f) - \mu_n(f)|^2] &\leq \frac{1}{N}\pa{\sum_{j=1}^n \Var_{\mu_j}(q_{j,n}(f)) + \fc{2\bar v_n \norm{f}_{L_p(\mu_n)}^2 \hat{c}_n
    }{N} }\\
    &\le  \frac{1}{N}\pa{\bar{v}_n + \frac{2\bar{v}_n\norm{f}_{L_p(\mu_n)}^2 \hat{c}_n }{N}} \label{simplified 2.1}
\end{align}
where 
\begin{equation}
\nonumber
    \hat{c}_n(p) = \sum_{j=1}^{n-1} c_{j,n}(p)(2 + \norm{q_{j,j+1}(1) - 1}_{L_p(\mu_j)}),
    \end{equation}
    and $c_{j,n}(p)$ is a constant such that for all $f\in B(E)$, 
    \begin{equation*}
\max\bc{
\norm{q_{j,n}(f)^2}_{L_p(\mu_j)}, \norm{q_{j,n}(f)}_{L_p(\mu_j)}^2, \norm{q_{j,n}(f^2)}_{L_p(\mu_j)}} \leq c_{j,n}(p)\Vert f\Vert _{L_p(\mu_n)}^2
    \end{equation*}
    and
\begin{equation}
    \nonumber 
\bar{v}_n = \max_{k \leq n} \sup\bigg\{\sum_{j=1}^k \Var_{\mu_j}(q_{j,k}(f)) : \norm{f}_{L_p(\mu_k)} \leq 1 \bigg\}.
\end{equation}
\end{thm}

The following lemmas work to obtain a closed form for $c_{k,n}(p)$ by finding $L_p$ bounds for $q_{k,n}(f)$. In obtaining these bounds it is easier to work with $\hat{q}_{k,n}(f)$, noting that $q_{k-1,n}(f) = \bar{g}_{k-1,k}\hat{q}_{k,n}(P_n(f)) $. We will first show that by assuming a bound of the form 
\begin{equation} \label{ab}
    \Vert \hat{q}_{k-1,k}(f)\Vert _{L_2(\mu_{k-1})}^2 \leq \alpha\Vert f\Vert _{L_2(\mu_k)}^2 + \beta\mu_k(f)^2,
\end{equation}
with $\alpha < 1$ and $\beta \ge 1$, 
we can iteratively obtain $L_2$ bounds for $\hat{q}_{k,n}(f)$ (Lemma \ref{l:3.3}). Again by assuming the bound (\ref{ab}), we can obtain an analogous $L_p$ version (Lemma \ref{l:3.3*}). Similarly, iterating over the $L_p$ version of (\ref{ab}) yields an $L_p$ bound for $\hat{q}_{k,n}(f)$ (Lemma \ref{l:3.4}). Combining all of these results allows us to find a closed form of $c_{k,n}(2)$ (Lemma \ref{lem:p=2}) which we then apply in Corollary \ref{cor:cv} to find bounds for $\hat{c}_n$ and $\bar{v}_n$. The bound for $\hat c_n$ will depend on a hypercontractivity constant.
\begin{lem}[{\cite[Lemma 3.2, adapted]{schweizer}}] \label{l:3.3} Assume there exists a bound of the form \eqref{ab}, $\Vert \hat{q}_{k-1,k}(f)\Vert _{L_2(\mu_{k-1})}^2 \leq \alpha\Vert f\Vert _{L_2(\mu_k)}^2 + \beta\mu_k(f)^2$ for all $f$, with $\alpha < 1$, $\beta\ge1$.  Then for $1 \leq j < k \leq n$ we have the bounds
    \begin{equation}
        \nonumber
         \norm{\hat{q}_{j,k}(f)}_{L_2(\mu_j)}\leq \bigg(\frac{\beta}{1-\alpha}\bigg)^\frac{1}{2}\norm{f}_{L_2(\mu_k)}.
    \end{equation}
\end{lem}
\begin{proof} By iterating a bound of the form $\Vert \hat{q}_{k-1,k}(f)\Vert _{L_2(\mu_{k-1})}^2 \leq \alpha\Vert f\Vert _{L_2(\mu_k)}^2 + \beta\mu_k(f)^2$ we get,
\begin{align*}
    \Vert \hat{q}_{j,k}(f)\Vert ^2_{L_2(\mu_j)} &\leq \alpha^{k-j}\mu_k(f^2) + \beta\sum_{i=0}^{k-(j+1)}\alpha^i\mu_k(f)^2\\
    \shortintertext{using the fact that $\mu_{j+1+i} (\hat q_{j+1+i,k}(f))= \mu_{k}(f)$. Since $\mu_k(f^2) \geq \mu_k(f)^2$ and $\beta \ge 1$,}
    &\leq \beta \sum_{i=0}^{k-j}\alpha^i\mu_k(f^2)\\
    \shortintertext{and since $\alpha < 1$,}
    &\leq \frac{\beta}{1-\alpha}\mu_k(f^2)= \frac{\beta}{1-\alpha}\norm{f}_{L_2(\mu_k)}^2.
\end{align*}
Taking the square root of both sides yields the desired result.
    \end{proof}
For the rest of this section, we assume there exists a uniform upper bound $\gamma$ on the density ratios $\bar{g}_{k-1,k}(x) \leq \gamma$ for all $k$. 
\begin{lem}[{\cite[Lemma 3.3]{schweizer}}]   \label{l:3.3*}
Assume that \eqref{ab} holds for each $k$, $1 < k \leq n$, $f\in B(E)$, and $p \geq 1$, we have
\begin{equation}
    \nonumber
    \Vert \hat{q}_{k-1,k}(f)\Vert ^{2p}_{L_{2p}(\mu_{k-1})} \leq \alpha\gamma^{2p - 2} \Vert f\Vert ^{2p}_{L_{2p}(\mu_{k})} + \beta\gamma^{2p-2}\Vert f\Vert ^{2p}_{L_{p}(\mu_{k})}. 
\end{equation}
\end{lem}
Similarly to Lemma \ref{l:3.3}, we adapt the below result to allow for $\beta > 1$.
\begin{lem}[{\cite[Lemma 3.4, adapted]{schweizer}}]  \label{l:3.4} Assume that \eqref{ab} holds for each $k$, $1 < k \leq n$. Assume that $\alpha\gamma^{2p-2} < 1$, $\beta\ge 1$, and that for $\delta(p) \geq 1$ we have for $1 \leq j \leq k \leq n$ and $f \in B(E)$ the inequality 
\begin{equation*}
    \Vert \hat{q}_{j,k}(f)\Vert _{L_p(\mu_j)} \leq \delta(p) \Vert  f \Vert _{L_p(\mu_k)}.
\end{equation*}
Then we have
\begin{equation*}
    \Vert \hat{q}_{j,k}(f)\Vert _{L_{2p}(\mu_j)} \leq \delta(2p) \Vert  f \Vert _{L_{2p}(\mu_k)}
\end{equation*}
with 
\begin{equation*}
    \delta(2p) =\delta(p)\bigg(\frac{\beta\cdot\gamma^{2p-2}}{1-\alpha\gamma^{2p-2}}\bigg)^\frac{1}{2p}.
\end{equation*}    
\end{lem}
\begin{proof}
By iterating the bound in Lemma~\ref{l:3.3*} and defining $\theta = \alpha\gamma^{2p-2}$ we have
\begin{align*}
      \Vert \hat{q}_{j,k}(f)\Vert ^{2p}_{L_{2p}(\mu_{j})} &\leq \theta^{k-j}\Vert f\Vert ^{2p}_{L_{2p}(\mu_{k})} + \beta\gamma^{2p-2}\sum_{i=0}^{k-(j+1)}\theta^i\Vert \hat{q}_{j+1 +i,k}(f)\Vert ^{2p}_{L_p(\mu_{j+i+1})}\\
      \shortintertext{By assumption $\Vert \hat{q}_{j,k}(f)\Vert _{L_p(\mu_j)} \leq \delta(p) \Vert  f \Vert _{L_p(\mu_k)}$,}
      &\leq \theta^{k-j}\Vert f\Vert ^{2p}_{L_{2p}(\mu_{k})} + \beta\gamma^{2p- 2}\delta(p)^{2p}\Vert f\Vert _{L_p(\mu_k)}^{2p}\sum_{i=0}^{k-(j+1)}\theta^i\\
      \shortintertext{Since $\Vert f\Vert _{L_p(\mu_k)} \leq \Vert f\Vert _{L_{2p}(\mu_k)}$ and $\beta \geq 1$, $\gamma\ge 1$,}
      &\leq\beta\gamma^{2p- 2}\delta(p)^{2p}\Vert f\Vert _{L_{2p}(\mu_k)}^{2p}\sum_{i=0}^{k-j}\theta^i\\
      \shortintertext{Since $\theta < 1$,}
      &\leq\delta(p)^{2p}\frac{\beta\gamma^{2p- 2}}{1-\theta}\Vert f\Vert _{L_{2p}(\mu_k)}^{2p}.
\end{align*}
Recalling $\theta = \alpha\gamma^{2p-2}$ and raising both sides to the power of $\frac{1}{2p}$ we get the desired result. 
\end{proof}
\begin{lem}\label{lem:p=2} Assume that \eqref{ab} holds for each $k$, $1 < k \leq n$. For $p > 2$, $f \in B(E)$ we have that, 
\begin{equation}
    \nonumber
    \max\bc{
    \norm{q_{j,k}(f)^2}_{L_p(\mu_j)}, \norm{q_{j,k}(f)}_{L_p(\mu_j)}^2, \norm{q_{j,k}(f^2)}_{L_p(\mu_j)}} \leq c_{j,k}(p)\Vert f\Vert _{L_p(\mu_k)}^2
\end{equation}
where
\begin{equation*}
    c_{j,k}(p) =\gamma^\frac{2p-1}{p}\delta(2p)^2\theta(p,p/2),
\end{equation*}
$\delta(2p)$ is as in Lemma~\ref{l:3.4} and $\theta(q,p)$ is a constant such that $\norm{P_k f}_{L^q(\mu_k)}\le \theta(q,p) \norm{f}_{L^p(\mu_k)}$ for each $f\ge 0$.
\end{lem}
\begin{proof}
We bound each of the above terms using Lemma \ref{l:3.3} and Lemma \ref{l:3.4}.
\begin{align*}
    \norm{q_{j,k}(f)^2}_{L_p(\mu_j)} 
    &= \norm{q_{j,k}(f)}_{L_{2p}(\mu_j)}^2\\
    &= \norm{\bar{g}_{j,j+1}\hat{q}_{j+1,k}(P_k(f))}_{L_{2p}(\mu_j)}^2\\
    &= \bigg(\int \bar{g}_{j,j+1}^{2p}\hat{q}_{j+1,k}(P_k(f))^{2p}\mu_j(dx) \bigg)^\frac{1}{p}\\ 
     \shortintertext{Since $\bar{g}_{k-1,k} \leq \gamma$,}
     &\leq\bigg(\gamma^{2p-1}\int \bar{g}_{j,j+1}\hat{q}_{j+1,k}(P_k(f))^{2p}\mu_j(dx) \bigg)^\frac{1}{p} \\
      &\leq\bigg(\gamma^{2p-1}\int \hat{q}_{j+1,k}(P_k(f))^{2p}\mu_{j+1}(dx) \bigg)^\frac{1}{p} \\
    &\leq \bigg(\gamma^\frac{2p-1}{2p}\norm{\hat{q}_{j+1,k}(P_k(f))}_{L_{2p}(\mu_{j+1})}\bigg)^2
    \shortintertext{By Lemma \ref{l:3.4} and Cauchy-Schwarz,}
    &\leq \bigg(\gamma^\frac{2p-1}{2p}\delta(2p)\Vert P_k(f)\Vert _{L_{2p}(\mu_k)}\bigg)^2\\
    &= \gamma^\frac{2p-1}{p}\delta(2p)^2\Vert P_k(f)^2\Vert _{L_{p}(\mu_k)}\\
    &\leq \gamma^\frac{2p-1}{p}\delta(2p)^2\Vert P_k(f^2)\Vert _{L_{p}(\mu_k)}\\
    \shortintertext{Now since $f^2 \ge 0$ is 
    non-negative we can use the local hypercontractivity inequality:}
    &\leq \gamma^\frac{2p-1}{p}\delta(2p)^2\theta(p,p/2)\Vert f^2\Vert _{L_{p/2}(\mu_k)}\\
    &=\gamma^\frac{2p-1}{p}\delta(2p)^2\theta(p,p/2)\Vert f\Vert _{L_{p}(\mu_k)}^2.
\end{align*}
By Lemma \ref{l:3.3} and the uniform upper bound $\bar{g}_{k-1,k}(x) \leq \gamma$ we have,
\begin{align*}
    \Vert q_{j,k}(f)\Vert _{L_p(\mu_{j})}^2 &=  \Vert \bar{g}_{j,j+1}\hat{q}_{j+1,k}(P_k(f))\Vert _{L_p(\mu_{j})}^2\\
    &\leq \gamma\delta(p)^2\Vert P_k(f)\Vert _{L_p(\mu_k)}^2\\
    &\leq \gamma\delta(p)^2\Vert f\Vert _{L_p(\mu_k)}^2.
\end{align*}
Lastly,  we have that 
\begin{align*}
    \Vert q_{j,k}(f^2)\Vert _{L_p(\mu_j)} &\leq \gamma^\frac{p-1}{p}\Vert \hat{q}_{j+1,k}(P_k(f^2))\Vert _{L_p(\mu_{j+1})}\\
    &\leq \gamma^\frac{p-1}{p}\delta(p)\Vert P_k(f^2)\Vert _{L_p(\mu_k)}\\
    &\leq \gamma^\frac{p-1}{p}\delta(p)\theta(p,p/2)\Vert f^2\Vert _{L_{p/2}(\mu_k)}\\
    &= \gamma^\frac{p-1}{p}\delta(p)\theta(p,p/2)\Vert f\Vert _{L_{p}(\mu_k)}.   \end{align*}
    Since $\delta(2p) \geq \delta(p)\ge 1$, $\gamma > 1$, and $\theta(p,p/2)\ge 1$, we have that $\gamma \delta(p)^2 \leq \gamma^\frac{2p-1}{p}\delta(2p)^2\theta(p,p/2)$ and $\gamma^\frac{p-1}{p} \delta(p)\theta(p,p/2) \leq \gamma^\frac{2p-1}{p}\delta(2p)^2\theta(p,p/2)$ and hence  $$\max\bc{\gamma \delta(p)^2 ,\delta^\frac{p-1}{p} \delta(p)\theta(p,p/2),\gamma^\frac{2p-1}{p}\delta(2p)^2\theta(p,p/2)} = \gamma^\frac{2p-1}{p}\delta(2p)^2\theta(p,p/2).  $$
\end{proof}
    \begin{corollary}\label{cor:cv} Assume that \eqref{ab} holds for each $k$, $1 < k \leq n$. Let $p > 2$, $f \in B(E)$ and defining $\hat{c}_n$ and $\bar{v}_n$ as in Theorem \ref{t:2.1} and $\delta(2p)$ is as in Lemma~\ref{l:3.4}, we have 
\begin{equation*}
    \hat{c}_n \leq n \cdot(3 + \gamma) \gamma^\frac{2p-1}{p}\delta(2p)^2\theta(p,p/2)
\end{equation*}
and
\begin{equation*}
    \bar{v}_n \leq n\cdot\gamma\cdot\bigg(\frac{\beta}{1-\alpha}\bigg).
\end{equation*}
\end{corollary}
\begin{proof}
Using our bounds established in Lemma \ref{l:3.3*} and Corollary \ref{lem:p=2}, we bound $\bar{v}_n$ and $\hat{c}_n$. 
\begin{align*}
    \hat{c}_n &= \sum_{j=1}^{n-1}c_{j,n}(p)\bigg(2 + \Vert q_{j,j+1}(1) -1\Vert _{L_p(\mu_j)}\bigg)\\
    \shortintertext{By Lemma \ref{lem:p=2}, }
    &= \sum_{j=1}^{n-1}\gamma^\frac{2p-1}{p}\delta(2p)^2\theta(p,p/2)\cdot\bigg(\frac{\beta}{1-\alpha\gamma^2} \bigg)^\frac{1}{2}\bigg(2 + \Vert q_{j,j+1} (1)-1\Vert _{L_2(\mu_j)}\bigg)\\
    &\leq (n-1) \cdot(3 + \gamma) \cdot \gamma^\frac{2p-1}{p}\delta(2p)^2\theta(p,p/2).
\end{align*}
Moreover, defining
\begin{align*}
    \hat{v}_n :&= \sup\bigg\{\sum_{j=1}^k \Var_{\mu_j}(q_{j,k}(f)) : \norm{f}_{L_p(\mu_k)} \leq 1 \bigg\}
\end{align*}
we have
\begin{align*}
\hat v_n 
    &\le \sup\bigg\{\sum_{j=1}^k \mu_j(q_{j,k}(f)^2) : \norm{f}_{L_2(\mu_k)} \leq 1 \bigg\}.\\
    \shortintertext{Using $\bar{g}_{k-1,k}(x) \leq \gamma$ and by Lemma \ref{l:3.3},}
    &\leq \sum_{j=1}^k \gamma\bigg(\frac{\beta}{1-\alpha}\bigg)\\
\shortintertext{This upper bound is increasing in $k$; hence}
 \bar{v}_n &= \max_{k\le n}\hat v_k \leq n\cdot\gamma\cdot\bigg(\frac{\beta}{1-\alpha}\bigg).
\end{align*}
\end{proof}

\section{Log-Sobolev Bounds \& Local Smoothing} \label{local mixing}In this section we consider the inter- and intra-mode variance of a mixture distribution. By limiting SMC to Markov chains whose generators are of the form $\langle f, \mathscr{L}f \rangle_{\mu} \leq \sum_{i=1}^m w_i\langle f, \mathscr{L}_{i}f \rangle_{\mu_i}$ we are able to obtain bounds for the intra-mode variance that depend only on the spectral gap of the decomposed chain. We also note that  Assumption \ref{a} provides us with log-Sobolev inequalities. Log-Sobolev bounds are stronger than Poincar\'e inequalities and in fact imply the Poincar\'e bounds $\Var_\pi(f) \leq c_{\textup{LSI}}\mathscr{E}_\pi(f,f)$ with the same constant. Therefore in this section we freely assume Poincar\'e inequalities. We will need to apply the stronger log-Sobolev bounds to give us hypercontractive results which depend only on local mixing. 

\subsection{Bounding the intra-mode variance}
\begin{lem}\label{l:markov}
    If the Markov semigroup $P_t$ is reversible with stationary distribution $\pi$, then
\begin{equation}
    \nonumber
    \frac{d^2}{dt^2}\Var_\pi(P_t f) \geq 0 .
\end{equation}
Hence, $-\ddd t \Var_\pi(P_tf) = -2\an{P_tf, \sL P_tf}$ is strictly decreasing.
\end{lem}
\begin{proof}
We compute
    \begin{align*}
         \frac{d^2}{dt^2}\Var_\pi(P_t f) &= \ddd t\pa{\ddd t \Var_\pi(P_t f) }\\
         & =\ddd t 2\an{P_t f, \sL P_t f}\\
         &= 2 (\an{\sL P_t f, \sL P_t f} + \an{P_t f , \sL^2 P_t f})\\
         \shortintertext{Since $\sL$ is self-adjoint,}
         &= 4\an{ \sL P_t f, \sL P_t f}\\
         &= 4\| \sL P_t f\|_{L^2(\pi)}^2 \geq 0 .
    \end{align*}
\end{proof}
\begin{lem} \label{intra}
Suppose that for a Markov generator $\mathscr{L}$ with stationary distribution $\pi = \sum_k w_k\pi_k $ there exists a decomposition $\langle f, \mathscr{L}f \rangle_{\pi} \leq \sum_{k=1}^m w_k\langle f, \mathscr{L}_{k}f \rangle_{\pi_k}$, where $\mathscr{L}_k$ has stationary measure $\pi_k$ and Poincar\'e constant $C_k$. Let $C^* = \max_k C_k$. Then for any function $f$, 
\begin{equation}
    \nonumber
   \sum_{k=1}^Mw_k \Var_{\pi_k}(P_Tf) \leq \frac{C^*}{2T}\Var_\pi(f).
\end{equation}
\end{lem}

\begin{proof}
    Since we have that
    \begin{align*}
        \ddd t \Var_\pi(P_tf) &= 2\an{P_tf, \sL P_tf},
    \end{align*}
we can consider
\begin{align*}
    \Var_\pi(f) &\geq 2 \int_0^T -\an{P_tf,\sL P_tf}_\pi dt \\
      \shortintertext{By Lemma \ref{l:markov}, since $-2\an{P_tf,\sL P_tf} $ is a strictly decreasing function of $t$,}
    &\geq -2 T\an{P_Tf,\sL P_Tf}_{\pi}\\
    &\geq 2T\sum_{k=1}^M-w_k \an{P_Tf,\sL_kP_Tf}_{\pi_k}\\
    &\geq \frac{2T}{C^*}\sum_{k=1}^Mw_k \Var_{\pi_k}(P_Tf).
\end{align*}
Hence $ \sum_{k=1}^Mw_k \Var_{\pi_k}(P_Tf) \leq \frac{C^*}{2T}\Var_\pi(f)$.
\end{proof}
\subsection{Bounding the total variance}
We now apply Lemma \ref{intra} to obtain bounds on $\Vert \hat{q}_{k-1,k}(f)\Vert _{L_2(\mu_{k-1})}$ which only depend on mixing within modes. First, note that $\Vert \hat{q}_{k-1,k}(f)\Vert _{L_2(\mu_{k-1})}^2  = \Var_{\mu_{k-1}}(\hat{q}_{k-1,k}(f)) + \mu_{k-1}(\hat{q}_{k-1,k}(f))^2 $. By the law of total variance, we will show that we can decompose the variance into the sum of its inter- and intra-mode variances, 
\begin{equation}
\label{e:inter-intra}
\Var_{\mu_{k-1}}(\hat{q}_{k-1,k}(f)) = \sum_{i=1}^{\blu{M_{k-1}}} \blu{w_{k-1}^{(i)}}\Var_{\mu_{k-1}^{(i)}}(\hat{q}_{k-1,k}(f)) + \sum_{i=1}^{\blu{M_{k-1}}} \blu{w_{k-1}^{(i)}} \bigg(\mu_{k-1}^{(i)}\big(P_{k-1}^{t_{k-1}}(\bar{g}_{k-1,k}f)\big) - \mu_{k-1}\big(\bar{g}_{k-1,k}f\big) \bigg)^2.
\end{equation}
We prove and make use of the above decomposition in the following section (Section \ref{mixing times}). We first show in the following Corollary to Lemma \ref{intra} that the intra-mode variance in \eqref{e:inter-intra} can be bounded. 
\begin{corollary}\label{Cor4} 
Suppose Assumption~\ref{a} holds. Then
    $$\sum_{i=1}^{\blu{M_{k-1}}} \blu{w_{k-1}^{(i)}}\Var_{\mu_{k-1}^{(i)}}(\hat{q}_{k-1,k}(f)) \leq \frac{C_{k-1}^* \cdot \gamma}{2T}\mu_k(f^2)$$
\end{corollary}

\begin{proof}
We have
    \begin{align*}
        \sum_{i=1}^{\blu{M_{k-1}}} \blu{w_{k-1}^{(i)}}\Var_{\mu_{k-1}^{(i)}}(\hat{q}_{k-1,k}(f)) &= \sum_{i=1}^{\blu{M_{k-1}}} \blu{w_{k-1}^{(i)}}\Var_{\mu_{k-1}^{(i)}}(P_{k-1}^{t_{k-1}}(\bar{g}_{k-1,k}f))\\
        \shortintertext{By Lemma \ref{intra},}
        &\leq \frac{C}{2t_{k-1}}\Var_{\mu_{k-1}}(\bar{g}_{k-1,k}f)\\
        &= \frac{C}{2t_{k-1}}\ba{\mu_{k-1}((\bar{g}_{k-1,k}f)^2) -\mu_k(\bar{g}_{k-1,k}f)^2}\\
        &\le 
        \frac{C \cdot \gamma}{2t_{k-1}}\mu_{{k-1}}(\bar{g}_{k-1,k}f^2) \\
        &\leq  \frac{C \cdot \gamma}{2t_{k-1}}\mu_{k}(f^2).
    \end{align*}
\end{proof}
Next we show that the inter-mode variance in \eqref{e:inter-intra} can be bounded. 
\begin{lem}\label{lem5} Given $\mu_{k-1}= \sum_{i=1}^{M_k} \blu{w_{k-1}^{(i)}} \mu_{k-1}^{(i)}$ we have that
\begin{equation*}
\sum_{i=1}^{\blu{M_{k-1}}} \blu{w_{k-1}^{(i)}} \bigg(\mu_{k-1}^{(i)}\big(P_{k-1}^{t_{k-1}}(\bar{g}_{k-1,k}f)\big) - \mu_{k-1}\big(\bar{g}_{k-1,k}f\big) \bigg)^2\leq \sum_{i=1}^{M} \frac{1}{w_i}\mu_k(f)^2.
\end{equation*}
\end{lem}
\begin{proof}
We have
     \begin{align*}
& \sum_{i=1}^{\blu{M_{k-1}}} \blu{w_{k-1}^{(i)}} \bigg(\mu_{k-1}^{(i)}\big(P_{k-1}^{t_{k-1}}(\bar{g}_{k-1,k}f)\big) - \mu_{k-1}\big(\bar{g}_{k-1,k}f\big) \bigg)^2\\
&= \sum_{i=1}^{\blu{M_{k-1}}} \blu{w_{k-1}^{(i)}}\bigg(\int_\Omega \bar{g}_{k-1,k}(x)f(x)\pa{\frac{d\mu_{k-1}^{(i)}P_{k-1}^{t_{k-1}}(x)}{d\mu_{k-1}(x)}-1}\mu_{k-1}(dx)\bigg)^2\\
&\leq \sum_{i=1}^{\blu{M_{k-1}}} \blu{w_{k-1}^{(i)}}\bigg(\int_\Omega \bar{g}_{k-1,k}(x)f(x)\ab{\frac{d\mu_{k-1}^{(i)}P_{k-1}^{t_{k-1}}(x)}{d\mu_{k-1}(x)}-1}\mu_{k-1}(dx)\bigg)^2\\
&\leq \sum_{i=1}^{\blu{M_{k-1}}} \blu{w_{k-1}^{(i)}}\bigg(\int_\Omega \bar{g}_{k-1,k}(x)f(x)\norm{\frac{d\mu_{k-1}^{(i)}P_{k-1}^{t_{k-1}}(x)}{d\mu_{k-1}(x)}-1}_{\sup}\mu_{k-1}(dx)\bigg)^2\\
\shortintertext{Then we have by Markov kernel contraction that $\norm{\frac{d\mu_{k-1}^{(i)}P_{k-1}^{t_{k-1}}(x)}{d\mu_{k-1}(x)}-1}_{\sup} \leq \norm{\frac{d\mu_{k-1}^{(i)}(x)}{d\mu_{k-1}(x)}-1}_{\sup}$:}
&\leq \sum_{i=1}^{\blu{M_{k-1}}} \blu{w_{k-1}^{(i)}}\bigg(\int_\Omega \bar{g}_{k-1,k}(x)f(x)\norm{\frac{d\mu_{k-1}^{(i)}(x)}{d\mu_{k-1}(x)}-1}_{\sup}\mu_{k-1}(dx)\bigg)^2\\
&\leq \sum_{i=1}^{\blu{M_{k-1}}} \blu{w_{k-1}^{(i)}}\bigg(\int_\Omega \bar{g}_{k-1,k}(x)f(x)\max\bc{\norm{\frac{d\mu_{k-1}^{(i)}(x)}{d\mu_{k-1}(x)}}_{\sup},1}\mu_{k-1}(dx)\bigg)^2\\
&=\sum_{i=1}^{\blu{M_{k-1}}} \blu{w_{k-1}^{(i)}}\bigg(\int_\Omega \bar{g}_{k-1,k}(x)f(x)\max\bc{\norm{\frac{d\mu_{k-1}^{(i)}(x)}{\sum_i\blu{w_{k-1}^{(i)}}d\mu_{k-1}^{(j)}(x)}}_{\sup},1}\mu_{k-1}(dx)\bigg)^2\\
&\leq \sum_{i=1}^{\blu{M_{k-1}}} \blu{w_{k-1}^{(i)}}\bigg(\int_\Omega \bar{g}_{k-1,k}(x)f(x)\max\bc{\norm{\rc{\blu{w_{k-1}^{(i)}}}\frac{d\mu_{k-1}^{(i)}(x)}{d\mu_{k-1}^{(i)}(x)}}_{\sup},1}\mu_{k-1}(dx)\bigg)^2\\
\shortintertext{Since $w_i \leq 1$,}
&=\sum_{i=1}^{\blu{M_{k-1}}} \blu{w_{k-1}^{(i)}}\bigg(\int_\Omega \bar{g}_{k-1,k}(x)f(x)\frac{1}{\blu{w_{k-1}^{(i)}}}\mu_{k-1}(dx)\bigg)^2\\
&=\sum_{i=1}^{\blu{M_{k-1}}} \frac{1}{\blu{w_{k-1}^{(i)}}}\mu_k(f)^2.
\end{align*}
\end{proof}

\section{Local Mixing Time Improvements} \label{mixing times}
We apply our intra- and inter-mode bounds to improve the mixing time results from Schweizer. 

\begin{lem}\label{l:af+b}
Let all assumptions from Lemma 3 hold and assume that $f$ is a non-negative function. Then 

\begin{equation}
    \nonumber
    \Vert \hat{q}_{k-1,k}(f)\Vert _{L_2(\mu_{k-1})}^2 \leq \lambda_{k-1}\Vert f\Vert _{L_2(\mu_{k})}^2 +  \blu{\beta_{k-1}}\mu_k(f)^2,
\end{equation}
where
\begin{equation}
    \nonumber
    \lambda_{k-1} = \frac{C_{k-1}^*\cdot \gamma}{2t_{k-1}} ,   \; \; \blu{\beta_{k-1}} = 1 +\sum_{i=1}^{\blu{M_{k-1}}}\frac{1}{\blu{w_{k-1}^{(i)}}}.
\end{equation}
\end{lem}
\begin{proof} By variance decomposition, 
\begin{align*}
&\Vert \hat{q}_{k-1,k}(f)\Vert _{L_2(\mu_{k-1})}^2 = \Vert P_{k-1}^{t_{k-1}}(\bar{g}_{k-1,k}f)\Vert _{L_2(\mu_{k-1})}^2\\
&=\int  (P_{k-1}^{t_{k-1}}(\bar{g}_{k-1,k}f) - \mu_{k-1}(\bar{g}_{k-1,k}f))^2\mu_{k-1}(dx) + \mu_{k-1}(\bar{g}_{k-1,k}f)^2\\
&=\sum_{i=1}^{\blu{M_{k-1}}}\blu{w_{k-1}^{(i)}}\int (P_{k-1}^{t_{k-1}}(\bar{g}_{k-1,k}f) - \mu_{k-1}( P_{k-1}^{t_{k-1}}(\bar{g}_{k-1,k}f)))^2\mu_{k-1}^{(i)}(dx) + \mu_{k-1}(\bar{g}_{k-1,k}f)^2\\
&=\sum_{i=1}^{\blu{M_{k-1}}}\blu{w_{k-1}^{(i)}}\int \big(P_{k-1}^{t_{k-1}}(\bar{g}_{k-1,k}f) - \mu_{k-1}^{(i)}(P_{k-1}^{t_{k-1}}(\bar{g}_{k-1,k}f))\big)^2\mu_{k-1}^{(i)}(dx)\\
&\quad + \sum_{i=1}^{\blu{M_{k-1}}}\blu{w_{k-1}^{(i)}} w_i\big(\mathbb{E}_{\mu_{k-1}^{(i)}}[P_{k-1}^{t_{k-1}}(\bar{g}_{k-1,k}f)]-\mathbb{E}_{\mu_{k-1}}[P_{k-1}^{t_{k-1}}(\bar{g}_{k-1,k}f)]\big)^2 +\mu_k(f)^2\\
&=\sum_{i=1}^{\blu{M_{k-1}}}\blu{w_{k-1}^{(i)}}\Var_{\mu_{k-1}^{(i)}}(P_{k-1}^{t_{k-1}}(\bar{g}_{k-1,k}f)) + \sum_{i=1}^{\blu{M_{k-1}}}\blu{w_{k-1}^{(i)}}\big(\mathbb{E}_{\mu_{k-1}^{(i)}}[P_{k-1}^{t_{k-1}}(\bar{g}_{k-1,k}f)]-\mathbb{E}_{\mu_{k-1}}[P_{k-1}^{t_{k-1}}(\bar{g}_{k-1,k}f)]\big)^2 \\
&\quad+ \mu_k(f)^2\\
\shortintertext{Applying Corollary \ref{Cor4} and Lemma \ref{lem5},}
&\leq \frac{C_{k-1}^*\cdot\gamma}{2t_{k-1}}\mu_k(f^2) + \pa{1 + \sum_{i=1}^{\blu{M_{k-1}}} \frac{1}{\blu{w_{k-1}^{(i)}}}} \mu_k(f)^2.
\end{align*}
\end{proof}
To be in correspondence with Lemma \ref{l:3.3} we require that $\lambda_{k}$ and $\beta_k$ are uniformly upper bounded for all $k$. 
We now introduce $\alpha \in (0,1)$ as an upper bound for the mixing at any particular level. The following is then immediate from Lemma~\ref{l:af+b}.
\begin{corollary}\label{c:mt}
\label{prop_time} Given $p \geq 2$, $\alpha \in (0,1)$. \blu{Let $\beta_{k-1} = 1 +\sum_{i=1}^{\blu{M_{k-1}}}\frac{1}{\blu{w_{k-1}^{(i)}}}$, as in Lemma \ref{l:af+b}. Assume there exists a uniform bound $\beta > 1 $ such that $\max_{1\leq k \leq n}\beta_k \leq \beta$.} If for all $1 \leq k \leq n$,
\begin{equation*}
    t_k \geq \frac{C^*_k\gamma}{2\cdot\alpha},
\end{equation*}
then we have that for all $1 \leq k \leq n$, 
\begin{equation*}
    \lambda_{k} \leq \alpha
\end{equation*}
which yields for all $k$, 
\begin{equation*}
    \Vert \hat{q}_{k-1,k}(f)\Vert _{L_2(\mu_{k-1})}^2 \leq \alpha\Vert f\Vert _{L_2(\mu_{k})}^2 +  \beta\mu_k(f)^2.
\end{equation*}
\end{corollary}
\section{Hypercontractivity}
\label{s:hypercontractivity}
Up to now we have only used the weaker Poincar\'e inequalities. In this section, we make use of the log-Sobolev inequalities to give us a hypercontractive result. Hypercontractivity allows us to exploit the smoothing properties of a Markov kernel to obtain bounds of the form \begin{align}\Vert P_tf\Vert _{L^q} &\leq \Vert f\Vert _{L^p},
\label{e:hc}
\end{align}
with $q > p >1$. The general idea is that $\Vert P_tf\Vert _{L^q} \rightarrow \Vert f\Vert _{L^1}$ as $t \rightarrow \infty$. 
Hence, given $p,q$, for some time $t$ we should be able to obtain \eqref{e:hc}, and given $p, t$, we have $\Vert P_tf\Vert _{L^{q(t)}} \leq \Vert P_tf\Vert _{L^p}$, for some function $q(t)\to \infty$ as $t\to \infty$.
Importantly, we show that we can obtain these bounds given only local mixing of $P_t$ up to a multiplicative constant. This allows us to obtain bounds in Lemma \ref{lem:p=2} that only require local mixing. 

\begin{lem}\label{hyper} Let $P_t$ be a reversible Markov process with stationary distribution $\pi = \sum_k w_k \pi_k$. Let 
$q(t) = 1+(p-1)e^{2t/C^*}$ where $C^*=\max_k c_k$.
Assume that the following hold. 
\begin{enumerate}
    \item (Assumption \ref{a}, part 2) There exists a decomposition of the form, 
    $$\langle f, \mathscr{L}f \rangle_\pi \leq \sum_{k=1}^m w_i \langle f, \mathscr{L}_k f \rangle_{\pi_k}.$$
    \item For each $\pi_k$ there exists a log-Sobolev inequality of the form, 
    $$\Ent{f^2}{\pi_k} \leq 2c_k\cdot\mathscr{E}_{\pi_k}(f,f).$$
\end{enumerate}

Let $f>0$ and $w^*=\min_k w_k$.
Then $\fc{\norm{P_tf}_L^{q(t)}(\pi)}{(w^*)^{\rc{q(t)}}}$ is a non-increasing function in $t$:
\begin{equation*}
    \frac{\Vert P_tf\Vert _{L^{q(t)}(\pi)}}{(w^*)^\frac{1}{q(t)}} \leq  \frac{\Vert P_0f\Vert _{L^{q(0)}(\pi)}}{(w^*)^\frac{1}{q(0)}} = \frac{\Vert f\Vert _{L^p(\pi)}}{(w^*)^\frac{1}{p}} 
\end{equation*}
and
\[
\Vert P_t f\Vert_{L^{q(t)}(\pi)} \le \theta(q(t),p) \Vert f \Vert_{L^p(\pi)}
\]
where
$\theta(q,p) = \pf{1}{w^*}^{\frac{1}{p}-\frac{1}{q}}.$

\end{lem}

\begin{proof} Note first that in our notation $P_tf^q = (P_tf)^q$. The result holds if $\frac{\Vert P_tf\Vert _{L^{q(t)}(\pi)}}{(w^*)^\frac{1}{q(t)}}$ is a decreasing function of $q(t)$, which is equivalent to showing, 
$$\frac{\partial}{\partial t}\log(\frac{\Vert P_tf\Vert _{L^{q(t)}(\pi)}}{(w^*)^\frac{1}{q(t)}}) \leq 0. $$

We will make use of the following results. 
\begin{enumerate}
    \item We have a decomposition for entropy given by
    \begin{equation}\label{ent}
    \Ent{f^2}{\pi} = \sum_{k=1}^m w_k \Ent{f^2}{\pi_k} + \Ent{\mathbb{E}_{\pi_k}[f^2]}{}.
    \end{equation}
    \item We have the following derivative (see below), 
\begin{equation}\label{log-deriv}
    \frac{\partial}{\partial t}\log(\mathbb{E}_\pi[P_tf^{q(t)}])=\frac{q'(t)}{q(t)\mathbb{E}_\pi[P_tf^{q(t)}]}\mathbb{E}_\pi[\log(P_tf^{q(t)})P_tf^{q(t)}]+\frac{q(t)}{\mathbb{E}_\pi[P_tf^{q(t)}]}\mathbb{E}_\pi[P_tf^{q(t)-1}\mathscr{L}P_tf]
\end{equation}
    \item (\cite[Lemma 8.11]{vanHandel})  For $q>1$ and a reversible Markov process $P_t$ with generator $\sL$ we have
\begin{equation}\label{8.11}
    \langle f^{q-1},\mathscr{L} f \rangle_\pi \leq \frac{4(q-1)}{q^2}\langle f^\frac{q}{2}, \mathscr{L}f^\frac{q}{2}\rangle_\pi.
\end{equation}
\end{enumerate}
We show the computations for the derivative (\ref{log-deriv})
\begin{align*}
    \frac{\partial}{\partial t}\log(\mathbb{E}_\pi[P_tf^{q(t)}])&= \frac{1}{\mathbb{E}_\pi[P_tf^{q(t)}]}\mathbb{E}_\pi\left[\frac{\partial}{\partial t}P_tf^{q(t)}\right]\\
    &= \frac{1}{\mathbb{E}_\pi[P_tf^{q(t)}]}\mathbb{E}_\pi\left[\frac{\partial}{\partial t}e^{q(t)\log(P_tf)}\right]\\
    &= \frac{1}{\mathbb{E}_\pi[P_tf^{q(t)}]}\mathbb{E}_\pi\left[\frac{\partial}{\partial t}\bigg(q(t)\log(P_tf)\bigg)P_tf^{q(t)}\right]\\
    &= \frac{1}{\mathbb{E}_\pi[P_tf^{q(t)}]}\mathbb{E}_\pi\left[\bigg(q'(t)\log(P_tf)+\frac{q(t)}{P_tf}\mathscr{L}P_tf\bigg)P_tf^{q(t)}\right]\\
     &= \frac{q'(t)}{\mathbb{E}_\pi[P_tf^{q(t)}]}\mathbb{E}_\pi[\log(P_tf)P_tf^{q(t)}]+\frac{q(t)}{\mathbb{E}_\pi[P_tf^{q(t)}]}\mathbb{E}_\pi[P_tf^{q(t)-1}\mathscr{L}P_tf]\\
     &=\frac{q'(t)}{q(t)\mathbb{E}_\pi[P_tf^{q(t)}]}\mathbb{E}_\pi[\log(P_tf^{q(t)})P_tf^{q(t)}]+\frac{q(t)}{\mathbb{E}_\pi[P_tf^{q(t)}]}\mathbb{E}_\pi[P_tf^{q(t)-1}\mathscr{L}P_tf].
\end{align*}
Now we consider the following, 
\begin{align*}
    &\frac{\partial}{\partial t}\log(\frac{\Vert P_tf\Vert _{L^{q(t)}(\pi)}}{(w^*)^\frac{1}{q(t)}})\\
    &= \frac{\partial}{\partial t}\log(\Vert P_tf\Vert _{L^{q(t)}(\pi)}) - \frac{\partial}{\partial t}\log((w^*)^\frac{1}{q(t)})\\
    &= \frac{\partial}{\partial t}\frac{\log(\mathbb{E}_\pi[P_tf^{q(t)}])}{q(t)} - \log(w^*)\frac{\partial}{\partial t}\frac{1}{q(t)}\\
        &= -\frac{q'(t)}{q(t)^2}\log(\mathbb{E}_\pi[P_tf^{q(t)}]) + \frac{1}{q(t)}\frac{\partial}{\partial t}\log(\mathbb{E}_\pi[P_tf^{q(t)}]) + \log(w^*)\frac{q'(t)}{q(t)^2}\\
        &= -\frac{q'(t)}{q(t)^2}\log(\mathbb{E}_\pi[P_tf^{q(t)}]) + \frac{q'(t)}{q(t)^2\mathbb{E}_\pi[P_tf^{q(t)}]}\mathbb{E}_\pi[\log(P_tf^{q(t)})P_tf^{q(t)}]\\
        &\quad +\frac{1}{\mathbb{E}_\pi[P_tf^{q(t)}]}\mathbb{E}_\pi[P_tf^{q(t)-1}\mathscr{L}P_tf] + \log(w^*)\frac{q'(t)}{q(t)^2}\\
        &=\frac{q'(t)}{q(t)^2\mathbb{E}_\pi[P_tf^{q(t)}]}\bigg(\mathbb{E}_\pi[P_tf^{q(t)}\log(P_tf^{q(t)})]-\mathbb{E}_\pi[P_tf^{q(t)}]\log(\mathbb{E}_\pi[P_tf^{q(t)}])\bigg)\\
        &\quad + \frac{1}{\mathbb{E}_\pi[P_tf^{q(t)}]}\mathbb{E}_\pi[P_tf^{q(t)-1}\mathscr{L}P_tf] + \log(w^*)\frac{q'(t)}{q(t)^2}\\
        &= \frac{q'(t)}{q(t)^2\mathbb{E}_\pi[P_tf^{q(t)}]}\Ent{P_tf^{q(t)}}{\pi}+ \frac{1}{\mathbb{E}_\pi[P_tf^{q(t)}]}\mathbb{E}_\pi[P_tf^{q(t)-1}\mathscr{L}P_tf] + \log(w^*)\frac{q'(t)}{q(t)^2}\\
        \shortintertext{By (\ref{8.11}),}
        &\leq \frac{q'(t)}{q(t)^2\mathbb{E}_\pi[P_tf^{q(t)}]}\Ent{P_tf^{q(t)}}{\pi}+\frac{4(q(t) - 1)}{q(t)^2\mathbb{E}_\pi[P_tf^{q(t)}]}\langle P_tf^\frac{q(t)}{2}, \mathscr{L}P_tf^\frac{q(t)}{2} \rangle_\pi + \log(w^*)\frac{q'(t)}{q(t)^2}\\
        \shortintertext{By the first assumption,}
        &\leq \frac{q'(t)}{q(t)^2\mathbb{E}_\pi[P_tf^{q(t)}]}\Ent{P_tf^{q(t)}}{\pi}+\frac{4(q(t) - 1)}{q(t)^2\mathbb{E}_\pi[P_tf^{q(t)}]}\sum_k w_k\langle P_tf^\frac{q(t)}{2}, \mathscr{L}_kP_tf^\frac{q(t)}{2} \rangle_{\pi_k} + \log(w^*)\frac{q'(t)}{q(t)^2}\\
        \shortintertext{By the second assumption (log-Sobolev inequality),}
        &\leq \frac{q'(t)}{q(t)^2\mathbb{E}_\pi[P_tf^{q(t)}]}\Ent{P_tf^{q(t)}}{\pi}+\frac{4(q(t) - 1)}{q(t)^2\mathbb{E}_\pi[P_tf^{q(t)}]}\sum_k w_k\bigg(-\frac{1}{2c_k}\Ent{P_tf^{q(t)}}{\pi_k}\bigg) + \log(w^*)\frac{q'(t)}{q(t)^2}\\
        \shortintertext{By the entropy decomposition (\ref{ent}),}
        &=\frac{q'(t)}{q(t)^2\mathbb{E}_\pi[P_tf^{q(t)}]}\bigg(\sum_k w_k \Ent{P_tf^{q(t)}}{\pi_k} + \Ent{\mathbb{E}_{\pi_k}[P_tf^{q(t)}]}{}\bigg)\\
        &\quad +\frac{4(q(t) - 1)}{q(t)^2\mathbb{E}_\pi[P_tf^{q(t)}]}\sum_k w_k\bigg(-\frac{1}{2c_k}\Ent{P_tf^{q(t)}}{\pi_k}\bigg) + \log(w^*)\frac{q'(t)}{q(t)^2}\\
        \shortintertext{Rearranging terms and applying $\Ent{\mathbb{E}_{\pi_k}[P_tf^{q(t)}]}{} = \sum_k w_k \mathbb{E}_{\pi_k}[P_tf^{q(t)}]\log(\frac{\mathbb{E}_{\pi_k}[P_tf^{q(t)}]}{\sum_kw_k\mathbb{E}_{\pi_k}[P_tf^{q(t)}]})$ yields}
        &= \frac{q'(t)}{q(t)^2\mathbb{E}_\pi[P_tf^{q(t)}]}\sum_kw_k\bigg(1 - \frac{4(q(t)-1)}{2c_kq'(t)}\bigg)\Ent{P_tf^{q(t)}}{\pi_k} \\
        &\quad + \frac{q'(t)}{q(t)^2\mathbb{E}_\pi[P_tf^{q(t)}]}\sum_k w_k \mathbb{E}_{\pi_k}[P_tf^{q(t)}]\log(\frac{\mathbb{E}_{\pi_k}[P_tf^{q(t)}]}{\sum_kw_k\mathbb{E}_{\pi_k}[P_tf^{q(t)}]})+ \log(w^*)\frac{q'(t)}{q(t)^2}.
\end{align*}
We stop here to note that currently we have, 
\begin{align*}
     \frac{\partial}{\partial t}\log(\frac{\Vert P_tf\Vert _{L^{q(t)}(\pi)}}{(w^*)^\frac{1}{q(t)}}) &\leq \underbrace{\frac{q'(t)}{q(t)^2\mathbb{E}_\pi[P_tf^{q(t)}]}}_{\geq 0}\sum_kw_k\bigg(1 - \frac{4(q(t)-1)}{2c_kq'(t)}\bigg)\underbrace{\Ent{P_tf^{q(t)}}{\pi_k}}_{\geq 0}\\
     &\quad+\underbrace{\frac{q'(t)}{q(t)^2\mathbb{E}_\pi[P_tf^{q(t)}]}}_{\geq 0}\sum_k w_k \underbrace{\mathbb{E}_{\pi_k}[P_tf^{q(t)}]}_{\geq 0}\log(\frac{\mathbb{E}_{\pi_k}[P_tf^{q(t)}]}{\sum_kw_k\mathbb{E}_{\pi_k}[P_tf^{q(t)}]})\\
     &\quad+ \log(w^*)\frac{q'(t)}{q(t)^2}
\end{align*}
Therefore, since $\log(x)$ is strictly increasing, we have that, 
\begin{align*}
    &\frac{q'(t)}{q(t)^2\mathbb{E}_\pi[P_tf^{q(t)}]}\sum_k w_k \mathbb{E}_{\pi_k}[P_tf^{q(t)}]\log(\frac{\mathbb{E}_{\pi_k}[P_tf^{q(t)}]}{\sum_kw_k\mathbb{E}_{\pi_k}[P_tf^{q(t)}]}) \\
    &\leq \frac{q'(t)}{q(t)^2\mathbb{E}_\pi[P_tf^{q(t)}]}\sum_k w_k \mathbb{E}_{\pi_k}[P_tf^{q(t)}]\log(\frac{\mathbb{E}_{\pi_k}[P_tf^{q(t)}]}{w_k\mathbb{E}_{\pi_k}[P_tf^{q(t)}]})\\
    &\leq\frac{q'(t)}{q(t)^2\mathbb{E}_\pi[P_tf^{q(t)}]}\sum_k w_k \mathbb{E}_{\pi_k}[P_tf^{q(t)}]\log(\frac{1}{w^*})\\
    &= \frac{q'(t)}{q(t)^2\mathbb{E}_\pi[P_tf^{q(t)}]}\mathbb{E}_{\pi}[P_tf^{q(t)}]\log(\frac{1}{w^*})\\
    &= -\frac{q'(t)}{q(t)^2}\log(w^*)
\end{align*}
Applying this to our inequality for the derivative yields, 
\begin{align*}
     \frac{\partial}{\partial t}\log(\frac{\Vert P_tf\Vert _{L^{q(t)}(\pi)}}{(w^*)^\frac{1}{q(t)}})& \leq \frac{q'(t)}{q(t)^2\mathbb{E}_\pi[P_tf^{q(t)}]}\sum_kw_k\bigg(1 - \frac{4(q(t)-1)}{2c_kq'(t)}\bigg)\Ent{P_tf^{q(t)}}{\pi_k}\\ 
     &\quad- \frac{q'(t)}{q(t)^2}\log(w^*)+ \frac{q'(t)}{q(t)^2}\log(w^*)\\ 
     &= \frac{q'(t)}{q(t)^2\mathbb{E}_\pi[P_tf^{q(t)}]}\sum_kw_k\bigg(1 - \frac{4(q(t)-1)}{2c_kq'(t)}\bigg)\Ent{P_tf^{q(t)}}{\pi_k}\\
     &\leq \underbrace{\frac{q'(t)}{q(t)^2\mathbb{E}_\pi[P_tf^{q(t)}]}}_{\geq 0}\cdot\bigg(1 - \frac{4(q(t) - 1)}{2c^*q'(t)}\bigg)\underbrace{\sum_kw_k\Ent{P_tf^{q(t)}}{\pi_k}}_{\geq 0}. 
\end{align*}
We can solve for when the above expression is 0 by solving the ODE $1 - \frac{4(q(t) - 1)}{2c^*q'(t)} = 0$. 
The solution to this ODE with initial value $q(0) = p$ is $q(t) = 1+ (p-1) e^{2t/c^*}$. Therefore we have that, $\frac{\Vert P_tf\Vert _{L^{q(t)}(\pi)}}{(w^*)^\frac{1}{q(t)}}$ is a decreasing function of $q(t) = 1+ (p-1) e^{2t/C^*}$, and so 
$$\frac{\Vert P_tf\Vert _{L^{q(t)}(\pi)}}{(w^*)^\frac{1}{q(t)}} \leq \frac{\Vert f\Vert _{L^p(\pi)}}{(w^*)^\frac{1}{p}}. .$$
\end{proof}
In particular, for $q = 4, p = 2$ we have that for $t = \frac{c^*\log(3)}{2}$, 
$$\Vert P_tf\Vert _{L^4(\pi)} \leq \frac{(w^*)^{1/4}}{(w^*)^{1/2}}\Vert f\Vert _{L^2(\pi)}= \frac{1}{(w^*)^\frac{1}{4}}\Vert f\Vert _{L_2(\pi)}$$
\section{Proof of Main Result} \label{s:proofs}
In this section we will provide state and prove a complete version of the main theorem.

  We note that for $p=2^k$ with $k \geq0$ the following function is defined recursively, 
    $$\delta(2p) = \delta(p)\bigg(\frac{\beta\cdot\gamma^{2p -2}}{1-\alpha\gamma^{2p-2}}\bigg)^\frac{1}{2p}$$
    with $\delta(1) = 1$, as in Lemma \ref{l:3.4} and Lemma \ref{l:3.3} respectively. 

\begin{thm}[\bf Complete MSE Bound] \label{t:mainVar_complete} Suppose Assumption~\ref{a} holds.
Then for all $\epsilon > 0$ with $0<\alpha < \frac{1}{\gamma^{2p-2}}$, $\beta = 1 + \blu{\max_{1\leq k \leq n} \sum_{i=1}^{M_{k}} \frac{1}{w_{k}^{(i)}}} > 1$, \blu{$w^* = \min_{k,i} w_{k}^{(i)}$, $p=2^k$}, $k \in \mathbb{N}_{>1}$, choosing
    \begin{enumerate}
        \item $ N \geq \max\bigg\{\frac{n}{\epsilon}\cdot\frac{\gamma\beta}{1-\alpha}(1 + \Vert f\Vert _{L^p(\mu_n)}^2
        + 2\Vert f - \mu_n(f)\Vert ^2_{\sup}),  \; 2 n (3+\gamma)\gamma^\frac{2p-1}{2p}\delta(2p)^2\frac{1}{\blu{(w^*)}^\frac{1}{2p}}\bigg\}$
        \item $t_k \geq \frac{C^*_k}{2}\max \bigg\{\frac{\gamma}{\alpha}, \log(\frac{p-1}{\frac{p}{2}-1})\bigg\}$
        \end{enumerate}
    yields
    \begin{equation*}
        \blu{\mathbb{E}[(\eta_n^N(f) - \mu_n(f))^2]}\leq \epsilon.
    \end{equation*}
\end{thm}

\begin{proof}\label{p:Var}
    The proof follows from the results of Section \ref{Var bounds} and Section \ref{local mixing}. 

    By Lemma \ref{l:2.2}, 
    \begin{equation}\mathbb{E}[(\eta_n^N(f) - \mu_n(f))^2] \leq 2\Var(\nu_n^N(f)) + 2\Vert f - \mu_n(f)\Vert ^2_{\sup,n}\Var(\nu_n^N(1))
    \label{e:t2-1}
    \end{equation}
    and by Theorem \ref{t:2.1},
    $$ \Var(\nu_n^N(f)) \leq \frac{1}{N}\pa{\bar{v}_n + \frac{2\bar{v}_n\norm{f}_{L_p(\mu_n)}^2 \hat{c}_n }{N}}.$$
    Furthermore, by Corollary \ref{cor:cv},
      $$ \hat{c}_n \leq n \cdot(3 + \gamma) \gamma^\frac{2p-1}{p}\delta(2p)^2\theta(p,p/2)$$
      and 
      $$\bar{v}_n \leq n \cdot \gamma\bigg(\frac{\beta}{1-\alpha}\bigg)$$
      when $\alpha, \beta$ satisfy the single-step bound \eqref{e:single-step}.
      Lastly we note that $\bar{c}_n = \max_{k\leq n} \hat{c}_k = \hat{c}_n$ and for Theorem \ref{t:2.1} to hold we require $N \geq 2\bar{c}_n$. Given this,
       \begin{equation}\Var(\nu_n^N(f)) \leq \frac{\bar{v}_n}{N}\big(1 + \norm{f}_{L_p(\mu_n)}^2\big).\label{e:t2-2}
    \end{equation}
    To ensure $N \geq 2\bar{c}_n$, by Lemma \ref{hyper}, with $\theta(p,\frac{p}{2})= \frac{1}{\blu{(w^*)}^\frac{1}{2p}}$ \blu{for $w^* = \min_{k,i}w_k^{(i)}$}, we require $p \le  q(t) = 1 + (\frac{p}{2}-1)e^\frac{2t}{c^*_k}$ for each $k$, i.e.
    $$t \ge \frac{\max_k c_k^*}{2}\log(\frac{p-1}{\frac{p}{2}-1}).$$
    Furthermore, from Lemma \ref{l:af+b} and Corollary \ref{c:mt}, the single-step bound \eqref{e:single-step} holds for $\alpha$ and \blu{$\beta_k \leq \beta = 1+\max_{1 \leq k \leq n}\sum_{i=1}^{M_k} \rc{w_k^{(i)}}$} when 
    $$t \geq \frac{\gamma \max_k C_k^*}{2\alpha}.$$
    Our result will be achieved only when
    $$t \geq \frac{\max_k C^*_k}{2}\max \bigg\{\frac{\gamma}{\alpha}, \log(\frac{p-1}{\frac{p}{2}-1})\bigg\}$$
    Lastly, we 
    combine \eqref{e:t2-1} and \eqref{e:t2-2} to obtain
    $$\Var(\eta_n^N(f)) \leq \frac{2\bar{v}_n}{N}\big(1 + \norm{f}_{L_p(\mu_n)}^2 + 2\Vert f - \mu_n(f)\Vert _{\sup}^2 \big).$$
    Setting this less than $\epsilon$ and solving for $N$ yields, 
    $$N \geq \frac{2\bar{v}_n}{\epsilon}\big(1 + \norm{f}_{L_p(\mu_n)}^2 + 2\Vert f - \mu_n(f)\Vert _{\sup}^2 \big).$$
    Therefore with the mixing time specified above and $$N \geq\max \bigg\{ \frac{2\bar{v}_n}{\epsilon}\big(1 + \norm{f}_{L_p(\mu_n)}^2 + 2\Vert f - \mu_n(f)\Vert _{\sup}^2 \big),2\bar{c}_n\bigg\},$$
    we have
    $\blu{\mathbb{E}[(\eta_n^N(f) - \mu_n(f))^2]}\le \epsilon$.
\end{proof}
\begin{proof}[Proof of Theorem \ref{t:mainVar} (Simplified MSE Bound)] 
This follows from using Theorem \ref{t:mainVar_complete} with parameters $k=2$ ($p=4$) and $\alpha = \frac{1}{2\gamma^6}$, and noting that the MSE does not change if we replace $f$ with $f-\mu_n(f)$. Because $\gamma \geq 1$ implies $1\leq\frac{1}{1-\frac{1}{2\gamma^k}} \leq 2 $, we therefore  have that
\begin{align*}
    \delta(8) &= \pf{\beta}{1-\frac{1}{2\gamma^6}}^{1/2}\pf{\beta\gamma^2}{1-\frac{1}{2\gamma^4}}^{1/4} \pf{\beta\gamma^6}{1-\frac{1}{2}}^{1/8}\\
    &\leq (2\beta)^{1/2}(2\beta\gamma^2)^\frac{1}{4}(2\beta\gamma^6)^\frac{1}{8}= \left(2\beta\right)^{7/8}\gamma^{5/4}.
\end{align*} 
Furthermore, noting that $\beta = 1 + \max_k\sum_{i=1}^{M_k}\frac{1}{w_k^{(i)}} \leq \frac{2\max_kM_k}{w^*}$ and substituting removes the $\beta$ term. 
\end{proof}

\bibliographystyle{imsart-number} 
\bibliography{bibliography}       


\end{document}